\documentclass[11pt]{article}
\usepackage{hyperref}

\usepackage[english]{babel}
\usepackage[left=0.5in,right=0.5in,top=1in,bottom=1.5in]{geometry} 
\usepackage{amssymb,amsmath,amsthm,amscd,ifthen}
\usepackage{mathtools}
\mathtoolsset{showonlyrefs}
\usepackage{graphicx}
\usepackage{xcolor}
\definecolor{ultramarine}{rgb}{0.07, 0.04, 0.56}
\definecolor{darkspringgreen}{rgb}{0.09, 0.45, 0.27}
\hypersetup{colorlinks,linkcolor={darkspringgreen},citecolor={ultramarine},urlcolor={ultramarine}}
\usepackage{bbm}
\usepackage{dsfont}
\usepackage{comment}
\usepackage{framed}
\usepackage{transparent}
\usepackage{color}
\usepackage{dirtytalk}
\usepackage{authblk}
\usepackage{framed}
\usepackage[authoryear]{natbib}
\bibpunct{(}{)}{;}{a}{,}{,}
\makeatletter
\def\BState{\State\hskip-\ALG@thistlm}
\makeatother

\newcommand{\R}{{\mathbb R}}
\def\P{\mathbb{P}}

\newcommand{\FF}{{\mathcal F}}
\newcommand{\eps}{\epsilon}
\DeclareMathOperator{\E}{\mathbb{E}}

\DeclareMathOperator{\supp}{supp}

\DeclareMathOperator{\Var}{Var} 
\DeclareMathOperator{\Ind}{\mathbbm{1}} 
\newcommand*{\F}{\mathcal{F}}

\DeclareMathOperator*{\argmin}{arg\,min}

\newcommand*{\A}{\mathcal{A}}

\newcommand{\T}{{\top}}

\newcommand{\cv}{\mathbf{c}}

\definecolor{cd60952}{RGB}{214,9,82}

\date{}

\newtheorem{Theorem}{Theorem}[section]
\newtheorem{Lemma}[Theorem]{Lemma}

\newtheorem{Corollary}[Theorem]{Corollary}
\newtheorem{Remark}[Theorem]{Remark}
\newtheorem{Example}[Theorem]{Example}

\DeclareMathOperator{\Tr}{Tr}

\DeclarePairedDelimiter{\norm}{\lVert}{\rVert}
\DeclarePairedDelimiter{\abs}{\lvert}{\rvert}
\DeclareMathOperator{\diam}{diam}

\DeclareMathOperator{\MOM}{MOM}
\DeclareMathOperator{\QOM}{QOM}
\newcommand*{\pmin}{p_{\min}}

\title{Robust $k$-means Clustering for Distributions with Two Moments}
\author[1]{Yegor Klochkov}
\author[2]{Alexey Kroshnin}
\author[3]{Nikita Zhivotovskiy}
\affil[1]{Cambridge-INET, Faculty of Economics, University of Cambridge, \href{mailto:yk376@cam.ac.uk}{yk376@cam.ac.uk}}
\affil[2]{HSE University and Institute for Information Transmission Problems, RAS \href{mailto:akroshnin@hse.ru}{akroshnin@hse.ru}}
\affil[3]{Google Research, Brain Team, \href{mailto:zhivotovskiy@google.com}{zhivotovskiy@google.com}}

\begin{document}

\maketitle

\abstract{
We consider the robust algorithms for the $k$-means clustering problem where a quantizer is constructed based on $N$ independent observations. Our main results are median of means based non-asymptotic excess distortion bounds {that} hold under the two bounded moments assumption in a general separable Hilbert space. In particular, our results extend the renowned asymptotic result of \cite{pollard1981strong} who showed that the existence of two moments is sufficient for strong consistency of an empirically optimal quantizer in $\R^d$. In a special case of clustering in $\R^d$, under two bounded moments, we prove matching (up to constant factors) non-asymptotic upper and lower bounds on the excess distortion, which depend on the probability mass of the lightest cluster of an optimal quantizer. Our bounds have the sub-Gaussian form, and the proofs are based on the versions of uniform bounds for robust mean estimators.
}

\section{Introduction}

Statistical (sample-based) $k$-means clustering is the classical form of quantization for probability measures. In this framework, given a distribution $P$ defined on a normed space $(E, \norm{\cdot})$ and an integer $k \ge 1$, one wants to find $A^* \subset E$ such that the \emph{distortion} 
\[
D(A) = \E\min\limits_{a \in A} \norm{X - a}^2 \quad\text{is minimized among all}\quad A \subset E,\; |A| = k.
\]
It is well known that if $(E, \norm{\cdot})$ is $\R^{d}$ with the Euclidean norm and if $\E \norm{X}^2 < \infty$ then this \emph{optimal quantizer} $A^*$ exists (see e.g., Theorem~1 in \citep{linder2002learning}), although it is not necessarily unique for $k \ge 2$. The value of the optimal distortion can be written as $D(A^*)$. In the statistical setup, the access to $P$ is achieved via $N$ independent observations $X_1, \ldots, X_N$ sampled according to $P$. Consider again the case of $\R^d$ and the Euclidean norm. The following renowned result due to \cite{pollard1981strong} states strong consistency of (any) \emph{empirically optimal quantizer}, which is defined by
\begin{equation}
\label{empquantizer}
    \hat{A} \in \argmin\limits_{A \subset \R^d, |A| = k} \frac{1}{N} \sum\limits_{i = 1}^N \min\limits_{a \in A} \norm{X_i - a}^2.
\end{equation}

\begin{Theorem}[Strong consistency of $k$-means \citep{pollard1981strong}]\label{thm_pollard}
    For any distribution $P$ such that $\E \norm{X}^2 < \infty$ and any integer $k \ge 1$, it holds that
    \[
    D(\hat{A}) - D(A^*) \overset{a.s.}{\to} 0,\quad \text{as}\quad N \to \infty.
    \]
\end{Theorem}

This consistency result was extended to the case where the space $(E, \norm{\cdot})$ is a general separable Hilbert space \citep*{biau2008performance, levrard2015nonasymptotic}. Clearly, the consistency alone does not provide any information on how many training samples are needed
to ensure that the excess distortion is below a certain level. Moreover, it does not allow the underlying distribution to be different for each sample size $N$.
Over the last three decades a lot of efforts have been made in order to prove non-asymptotic results for the \emph{excess distortion} $D(\hat{A}) - D(A^*)$ where the space is $\R^d$ or a general separable Hilbert space. We refer to various bounds in \citep*{bartlett1998minimax, linder2002learning, biau2008performance, graf2007foundations, maurer2010k, narayanan2010sample,  levrard2013fast, levrard2015nonasymptotic, fefferman2016testing, maurer2016vector} 
and the references therein. However, almost all the results were provided under the strong assumption that the domain is bounded. That is, it is usually assumed that $\norm{X} \le T$ almost surely where $X$ is distributed according to $P$ and $T > 0$ is a constant. This simple setup allows one to use the tools of Empirical Process Theory in order to prove results of the form (see, e.g., Theorem~2.1 by \cite*{biau2008performance}, where the space $(E, d)$ is assumed to be a separable Hilbert space)
\begin{equation}
\label{boundedcase}
    D(\hat{A}) - D(A^*) \lesssim T^2 \sqrt{\frac{k^2 + \log \frac{1}{\delta}}{N}},
\end{equation}
holding, with probability at least $1 - \delta$, for $\delta \in (0,1)$, where the notation $\lesssim$ suppresses absolute multiplicative constants. The question of general unbounded distributions is more challenging and has been studied less. The case where the vectors $X_i$ have well behaved exponential moments was analyzed in \citep{cadre2012holder}. Some results under less restrictive assumptions include: the uniform deviation bounds in \citep*{telgarsky2013moment, bachem2017uniform}; a \emph{sub-Gaussian} distortion bound in \citep*{brownlees2015empirical} for the so-called $k$-medians problem; and the results in \citep*{brecheteau2018robust} for trimmed quantizers. We discuss some of these results in more detail in what follows. However, we emphasize that in our particular setup the results we are aware of require the existence of at least \emph{four moments} (that is, $\E \norm{X}^4 < \infty$) compared to the minimal assumption under which the problem makes sense~--- $\E \norm{X}^2 < \infty$~--- which we are aiming for in this paper (this assumption is required to define the distortion $D(A^*)$). The question whether non-asymptotic results of the form \eqref{boundedcase} are possible under the minimal assumption $\E \norm{X}^2 < \infty$ (as in \citep{pollard1981strong}) appeared naturally in several papers (see, e.g., \citep{levrard2013fast}) but has not yet been addressed.

Our motivating example is the sub-Gaussian mean estimator introduced in \citep{lugosi2019sub}. Consider the situation where $E = \R^d$ with the Euclidean norm, set $\mu = \E X$, and assume that the covariance matrix $\Sigma = \E(X - \mu)(X - \mu)^\T$ exists. If $k = 1$, we obviously have that the optimal quantizer $A^*$ is actually the mean $\mu$. 
In this case, our problem boils down to the estimation of the mean of a random vector. It was shown by \citeauthor{lugosi2019sub} that there is an estimator (denoted by $\hat{a}$) such that, with probability at least $1 - \delta$,
\begin{equation}
\label{lugosimendelson}
    \E \norm{X - \hat{a}}^2 - \E \norm{X - \mu}^2 
    \lesssim \frac{\E \norm{X - \mu}^2 + \lambda_{\max}(\Sigma) \log \frac{1}{\delta} }{N},
\end{equation}
where $\lambda_{\max}(\Sigma)$ is the largest eigenvalue of the covariance matrix $\Sigma$, the expectation is taken only with respect to $X$, and $\hat{a} = \hat{a}(X_1, \ldots, X_N)$ is random. 
It is known that this bound is valid for the sample mean in the case where the underlying distribution is multivariate Gaussian. The bound \eqref{lugosimendelson} has some remarkable properties:
\begin{itemize}
    \item The dependence on $N$ is $O\left(\frac{1}{N}\right)$.
    
    \item It only requires the existence of two moments, that is $\E \norm{X}^2 < \infty$. We note that $\lambda_{\max}(\Sigma) \le \Tr(\Sigma) = \E \norm{X - \mu}^2$.
    
    \item It has the logarithmic dependence on the confidence, which is $\log \frac{1}{\delta}$ and corresponds to the situation where $\norm*{\mu - \hat{a}}$ has \emph{sub-Gaussian tails} (see, e.g., \citep{Vershynin2016HDP} for various equivalent definitions of sub-Gaussian distributions).
    
    \item Finally, even in the favorable bounded case where $\norm{X - \mu} \le T$ almost surely, the bound \eqref{lugosimendelson} does not scale as $T^2$ (compare it with the typical $k$-means bound \eqref{boundedcase}) but as $\E \norm{X - \mu}^2$ which can be much smaller than $T^2$.  
\end{itemize}

Therefore, extending the original question of whether the non-asymptotic excess distortion bounds are possible under $\E \norm{X}^2 < \infty$, it is natural to ask if one can prove a result of the form \eqref{lugosimendelson} for $k \ge 2$. Unfortunately, a fully general picture is much more subtle. In particular, even in the favorable bounded case for $k \ge 2$, lower bounds of order $\Omega\left(\frac{1}{\sqrt{N}}\right)$ are known (see, e.g., \citep{antos2005improved}) making the simple bound \eqref{boundedcase} sharp with respect to $N$.

Further, if $k = 1$ we observe that the right-hand side of \eqref{lugosimendelson} converges to zero as $N$ goes to infinity even if the underlying distribution $P$ depends on the sample size $N$ (see e.g., Example \ref{pminexample} below). Our only condition is that $\E \norm{X - \mu}^2 = \Tr(\Sigma)$ does not grow too fast as $N$ goes to infinity. Example~\ref{pminexample} below shows that the same is not true for general $k \ge 2$. 

Risk bounds having the sub-Gaussian form for heavy-tailed distributions have attracted a lot of attention recently. Among these advances are (almost) optimal results on mean estimation in various norms \citep{minsker2018uniform, lugosi2019near} and in classification \citep*{lecue2018robust}, covariance estimation \citep{mendelson2018robust} and robust regression \citep{minsker2019excess, lugosi2019regularization, lecue2017robust}. All the technical results in this area are based on different versions of the so-called \emph{median of means estimator}, which was first introduced and analyzed by \cite{nemirovsky1983problem} and independently in \citep*{alon1999space}. For the sake of completeness, let us recall this basic result.

Assume that $Y_1, \ldots, Y_N$ are independent random variables with the same mean $\mu$ and the same variance $\sigma^2$. Fix the confidence level $\delta$ and assume that $\ell = \left\lceil 8 \log \frac{1}{\delta}\right\rceil$ is such that $N = m \ell$, where $m$ is integer. Split the set $\{1, \ldots, N\}$ into $\ell$ blocks $I_1, \ldots, I_{\ell}$ of equal size such that $I_j = \{1 + m(j - 1), \ldots, mj\}$. Denote the \emph{median of means} (MOM for short) estimator by
\[
\hat{\mu} = \mathrm{Median}\left(\frac{\ell}{N}\sum\limits_{i \in I_1}Y_i, \ldots, \frac{\ell}{N} \sum\limits_{i \in I_\ell}Y_i\right).
\]
For this estimator we have the following sub-Gaussian behaviour. It holds, with probability at least $1 - \delta$, that
\[
\abs*{\hat{\mu} - \mu} \le \sigma\sqrt{\frac{32 \log \frac{1}{\delta}}{N}}.
\]

Returning to the question of $k$-means clustering and the inequalities of the form \eqref{lugosimendelson} for general $k \ge 2$, the following simple example inspired by \citeauthor*{bachem2017uniform} highlights some of the obstacles we have to handle.

\begin{Example}
\label{pminexample}
    Let $N$ be the sample size. Consider the real line $\R$, $k = 2$ and the distribution $P$ supported on $\{0, \sqrt{N}\}$ such that $P(\{0\}) = 1 - \frac{1}{N}$ and $P(\{\sqrt{N}\}) = \frac{1}{N}$. In this case we have $D^*(A) = 0$.
    
    One may easily see that, with constant probability, the value $\sqrt{N}$ is not among the observations $X_1, \ldots, X_N$. That obviously forces $\hat{A} = \{0\}$ and
    \[
    D(\hat{A}) - D(A^*) = \E X^2 = 1,
    \]
    which is not converging to zero as $N$ goes to infinity.
\end{Example}
In Section~\ref{sec:lowerbound} we significantly extend this construction. Of course, Example~\ref{pminexample} does not contradict the strong consistency result of Theorem~\ref{thm_pollard}. Although $\E X^2 = 1$, the distribution $P = P(N)$ changes with $N$, which is, of course, not allowed in Theorem~\ref{thm_pollard}. However, in the statistical learning theory literature the underlying distribution $P$ is usually allowed to be different for each value of $N$ as in Example \ref{pminexample}. This provides an additional motivation for our study. Our general bounds provide consistency even for some sequences of distributions changing with the sample size $N$.

\subsection*{On Voronoi cells and clustering}

From now on we assume that $(E, \norm{\cdot})$ is a separable Hilbert space with the inner product denoted by $\langle\cdot, \cdot \rangle$. 
Any quantizer $A = \{a_1, \ldots, a_k\}$ induces a partition of \(E\) into the so-called \emph{Voronoi cells}, which for \(a \in A\) consists of the points that have \(a\) as the closest point from \(A\). To avoid the uncertainty at the boundaries, we assume that the elements of each quantizer \( A = \{a_1, \dots, a_k \} \) are ordered, and define the cells for $j = 1, \ldots, k$,
\begin{align*}
    V_{j}(A) = \bigl\{x \in E \;:\; & \norm{x - a_j} < \norm{x - a_{j'}}, \; j' = 1, \dots, j - 1, \\
    & \norm{x - a_j} \le \norm{x - a_{j'}}, \; j' = j + 1, \dots, k \bigr\}.
\end{align*}
This way, we ensure that the cells are non-intersecting, and each of them is an intersection of \( k - 1\) open or closed half-spaces. 
Slightly abusing the notation, we sometimes write $V_{a_j}$ instead of $V_{j}$. 

We recall some properties of an \emph{optimal} quantizer under the assumption $\E \norm{X}^2 < \infty$. 
\begin{enumerate}
    \item\label{prop:existence} For any distribution $P$ with $\E \norm{X}^2 < \infty$ and any $k$ there exists an optimal $k$-elements quantizer $A^*$ (see \citep[Corollary~3.1]{fischer2010quantization}) Note that an optimal quantizer is not necessarily unique.

    \item\label{zero_intesec} For any optimal $A^*$ and $i \neq j$,
    \[
    P\left(\norm{X - a_j} = \norm{X - a_i}\right) = 0,
    \]
    which means that the measure of intersection of any two cells is zero, thus it does not matter to which cell the boundary points are assigned (see \citep{graf2007foundations}, Theorem~4.2.)
    
    \item The \emph{centroid condition} \citep{graf2007foundations}: for $j = 1, \ldots, k$,
    \begin{align}
        &\E \norm{X - a_j}^2 \Ind[X \in V_{j}] = \inf\limits_{a \in E} \E \norm{X - a}^2 \Ind[X \in V_{j}], \nonumber \\ 
        &a_j = \frac{\E X \Ind[X \in V_{j}]}{P(V_j)}, \;\; \text{whenever}\;\; P(V_j) > 0.
    \end{align}
    
    \item Once the support of $P$ consists of at least $k$ elements, there is a well-defined real number $M = M(P, k)$ such that for any optimal $A^*$,
    \begin{equation}
    \label{upperm}
        \norm{a} \le M \quad \text{for all}\quad a \in A^*.
    \end{equation} 
    We refer to the original proof of \citeauthor{pollard1981strong} or to Lemma~5.1 in \citep{fischer2010quantization}. 
    
    \item\label{prop:p_min} Due to Theorem~4.1 in \citep{graf2007foundations} provided that the support of $P$ consists of at least $k$ elements, there exists $\pmin > 0$ such that for any optimal $A^*$,
    \begin{equation}
    \label{pmin}
        \min\limits_{j}P(V_{j}(A^*)) \ge \pmin.
    \end{equation}
\end{enumerate}

Observe that the same conclusions work if we replace $P$ by its empirical counterpart $P_N$. 
In particular, a version of centroid condition is also valid for the empirically optimal quantizer defined by \eqref{empquantizer}. 
However, it is not true for a MOM based estimator in general.


\subsection*{Structure of the paper}

\begin{itemize}
    \item Section~\ref{sec:warmup} is devoted to a high probability excess distortion bound that holds in the case where a good guess on the localization radius of the optimal quantizer $A^*$ is available. The result generalizes naturally several known bounds for the empirically optimal quantizer in separable Hilbert spaces.
    
    \item Section~\ref{sec:pminbound} contains our main results. We show that there is a consistent median-of-means based estimator that gives the sub-Gaussian performance under our minimal moment assumption provided that a good guess on $\pmin$ is given and $\pmin N \to \infty$. We also prove a lower bound showing that our dependence on $\pmin$ and $N$ is sharp up to constant factors in the special case of $\R^d$.
    
    \item Finally, Section~\ref{sec:Generalcase} is devoted to the generalization of our main results. We show that it is possible to prove a slightly weaker bound using the procedure that does not require the knowledge of the parameters of $P$.
    
    \item Section~\ref{Discussions} is devoted to some final remarks.
\end{itemize}

\subsection*{Notation}

For $a, b \in \R$, we set $a \wedge b = \min\{a, b\}, a \lor b = \max\{a, b\}$ and for two real valued functions $f, g$, we write $f \lesssim g$ iff there is an absolute constant $c > 0$ such that $f \le c g$. 
We set $f \simeq g$ if $f \lesssim g$ and $g \lesssim f$. 
Given a probability measure $P$ let $P^{\otimes N}$ denote the measure which is $N$-times product of $P$. For the sake of simplicity, we always assume that $\log x$ is equal to $\log x\: \lor\: 1$. The indicator of an event $A$ is denoted by $\Ind[A]$. We also use the standard $O(\cdot)$, $\Omega(\cdot)$, $\Theta(\cdot)$ notation as well as $\mathrm{KL}(P, Q)$ and $\mathrm{TV}(P, Q)$ for Kullback--Leibler divergence and Total Variation distance between two measures $P$ and $Q$ (see, e.g., \citep{boucheron2013concentration}). The support of a measure $P$ is denoted by $\supp(P)$. For a normed space $(E,  \norm{\cdot})$ let $B_{R}$ denote the closed ball of radius $R$ centered at the origin. To avoid the measurability issues we use the standard convention for the supremum of stochastic processes (see Paragraph~2 in \citep{Talagrand2014}). Given the sample $X_{1}, \ldots, X_{N}$ sampled i.i.d.\ from $P$ and a function $f \colon E \to \R$ we denote $P_N f = P_N f(X) = \frac{1}{N} \sum\limits_{i = 1}^N f(X_i)$. In general, the symbol $P_N$ denotes the empirical measure.

We are interested in $L_2(P_N)$ space, and the corresponding covering number of a functional class $\mathcal{G}$ is denoted by $\mathcal{N}_{2}(\mathcal{G}, x, P_N)$, where $x$ is the corresponding radius (see e.g., \citep{Vershynin2016HDP} for more details on covering numbers).

\section{Simple Case: Known Magnitude of an Optimal Quantizer}
\label{sec:warmup}

In this section we provide our simplest result which can serve as a good illustration of the underlying techniques. In Sections~\ref{sec:pminbound} and~\ref{sec:Generalcase} we are focusing on sharpening our basic bound as well as weakening some of the assumptions. 

We first show a simple bound which holds in the situations where a good guess on $M$ is available (recall the property~\eqref{upperm}). The result of Theorem~\ref{thm:simpleupperm} below can be seen as a strengthening of Theorem~11 in \citep{brownlees2015empirical}.

\begin{Remark}
\label{importantremark}
    It is important to note that the boundedness of the vectors in the finite set $A^*$ has nothing in common with the boundedness of the observations $X_1, \ldots, X_N$. The latter can still be unbounded and the distribution $P$ can be heavy-tailed. 
\end{Remark}

We proceed with the main result of this section.
\begin{Theorem}
\label{thm:simpleupperm}
    Fix $\delta \in (0, 1)$.
    Let some $M$ satisfying \eqref{upperm} be known. There is an estimator $\hat{A}_{\delta, M}$ that depends on $M$ and $\delta$ such that, with probability at least $1 - \delta$,
    \[
    D(\hat{A}_{\delta, M}) - D(A^*) \lesssim M \left(M + \sqrt{\E \norm{X}^{2}}\right) \left(\frac{k}{\sqrt{N}} + \sqrt{\frac{\log \frac{1}{\delta}}{N}}\right).
    \]
\end{Theorem}

Let us now define the estimator that we use in Theorem~\ref{thm:simpleupperm}. 
Notice that minimizing \(\E \min_{a \in A} \norm{a - X}^{2} \) with respect to $A$ is equivalent to minimizing \(\E l_{A}(X)\), where
\[
    l_{A}(x) = \min_{a \in A} -2 \langle x, a\rangle + \norm{a}^{2} \, .
\]
Fix $1 \leq \ell \leq N$ and assume without loss of generality that $m = N / \ell$ is integer. Split the set $\{1, \ldots, N\}$ into $\ell$ blocks $I_1, \ldots, I_{\ell}$ of equal size such that $I_j = \{1 + (j - 1) m, \ldots, j m\}$. For any real-valued function $f$ and random variables $X_1, \ldots, X_N$ define
\begin{equation}\label{def:mom}
    \MOM(f) = \mathrm{Median}\left(\frac{\ell}{N} \sum\limits_{i \in I_1} f(X_i), \ldots, \frac{\ell}{N} \sum\limits_{i \in I_\ell} f(X_i)\right).
\end{equation}

Slightly abusing the notation we also set
\begin{equation}\label{def:Ak}
    \A^k = \left\{A \subset E \;:\; |A| \le k\right\} \quad \text{and} \quad 
    \A^k_M = \left\{A \in \A^k \;:\; \max\limits_{a \in A} \norm{a} \le M\right\}.
\end{equation}

\begin{framed}
    \textbf{The estimator of Theorem~\ref{thm:simpleupperm}}. Define
    \[
        \hat{A}_{\delta, M} = \arg\min_{A \in \A_{M}^k} \MOM(l_{A}),
    \]
    with the number of blocks \(\ell = 8 \left\lceil \log \tfrac{2}{\delta}\right\rceil + 1\). If there are many minimizers, we may choose any of them.
\end{framed}

The proof of Theorem~\ref{thm:simpleupperm} relies on the uniform bound for the median of means estimator. However, instead of restricting our attention to the medians only, we consider the \emph{quantiles of means} (QOM). That is, for a given level \( \alpha \in (0, 1)\),
\[
    \QOM_{\alpha}(f) = \mathrm{Quant}_{\alpha}\left( \frac{\ell}{N} \sum\limits_{i \in I_1} f(X_i), \ldots, \frac{\ell}{N} \sum\limits_{i \in I_\ell} f(X_i) \right),
\]
where \( \mathrm{Quant}_{\alpha}(x_1, \dots, x_\ell) = x^{(\lceil \alpha \ell \rceil)} \), for \( x^{(1)}, \dots, x^{(\ell)} \) being a non-decreasing rearrangement of the original sequence. For the sake of simplicity, we always assume that \( \ell \alpha \) is non-integer, such that the quantile is uniquely defined, and, in particular 
\( \mathrm{Quant}_{\alpha}(x_1, \dots, x_\ell) = - \mathrm{Quant}_{1 - \alpha}(-x_1, \dots, -x_\ell)\). It is usually enough to assume that $\ell$ is not even which can be always achieved by adding at most one extra block. Obviously, $\QOM_{\frac{1}{2}}$ corresponds to the median of means.

\begin{Lemma}
\label{lem:momuniform}
    Fix $\alpha \in (0, 1)$ and consider a separable class $\F$ of square integrable real-valued functions. Suppose, we have $\ell$ blocks and \( \ell \alpha \) is a non-integer. It holds that, with probability at least \(1 - e^{-\alpha^{2} \ell / 2}\),
    \begin{equation}
    \label{absvalueequation}
        \sup\limits_{f \in \F} (\E f - \QOM_{\alpha}(f)) 
        \le \frac{16}{\alpha} \E \sup\limits_{f \in \F} \left(\frac{1}{N} \sum_{i = 1}^{N} \eps_i f(X_i)\right) + \sqrt{\frac{8}{\alpha} \sup\limits_{f \in \F} \Var(f(X))\ \frac{\ell}{N}},
    \end{equation}
    as well as, with probability at least \( 1 - e^{-(1-\alpha)^2 \ell / 2} \),
    \begin{align}
    \label{absvalueequation_second}
        &\sup\limits_{f \in \F} (\QOM_{\alpha}(f) - \E f) \nonumber
        \\
        &\quad\quad \le \frac{16}{1 - \alpha} \E \sup\limits_{f \in \F} \left(\frac{1}{N} \sum_{i = 1}^{N} \eps_i f(X_i)\right) + \sqrt{\frac{8}{1 - \alpha} \sup\limits_{f \in \F} \Var(f(X))\ \frac{\ell}{N}},
    \end{align}
    where \( \eps_{1}, \dots, \eps_{N} \) are i.i.d.\ Rademacher signs. 
\end{Lemma}

\begin{Remark}
    In the case where \( \alpha \) is fixed, we can take \( \ell \simeq \log \tfrac{1}{\delta} \), so that with probability at least \( 1 - \delta \),
    \[
        \sup\limits_{f \in \F} \abs{\E f - \QOM_{\alpha}(f)} 
        \lesssim \; \E \sup\limits_{f \in \F}\left(\frac{1}{N} \sum_{i = 1}^{N} \eps_i f(X_i)\right) + \sqrt{\sup\limits_{f \in \F} \Var(f(X))\ \frac{\log\tfrac{1}{\delta}}{N}},
    \]
    where the first term represents the expectation of the empirical process, whereas the second term corresponds to a tail with the sub-Gaussian behavior. 
    Compare this inequality with Talagrand's inequality for empirical processes, where the assumption \(\sup\limits_{f \in \F} \abs{f(X)} \le C\) almost surely is needed (see Chapter~12 in \cite{boucheron2013concentration}).
\end{Remark}

As noticed by \cite{minsker2018uniform} (see equation~(2.7)) an inequality similar to \eqref{absvalueequation} of Lemma~\ref{lem:momuniform} for $\alpha = \frac{1}{2}$ follows from the proof of Theorem~2 in \citep*{lecue2018robust}. However, to the best of our knowledge, Lemma~\ref{lem:momuniform} in this form is not presented explicitly in the literature. We provide its proof in the appendix for the sake of completeness. 

\begin{proof}[Proof of Theorem~\ref{thm:simpleupperm}]
    \textbf{Step 1.} 
    First, we provide the high probability part of the analysis. Observe that
    \begin{align*}
    D(\hat{A}_{\delta, M}) - D(A^*) &= \E l_{\hat{A}_{\delta, M}} - \E l_{A^*}
    \\
    &\le \E l_{\hat{A}_{\delta, M}} - \MOM(l_{\hat{A}_{\delta, M}}) + \MOM(l_{A^*}) - \E l_{A^*}
    \\
    &\le 2 \sup\limits_{A \in \A_M^k} \abs*{\E l_{A} - \MOM(l_{A})},
    \end{align*}
    where we used $\MOM(l_{A^*}) \ge \MOM(l_{\hat{A}_{\delta, M}})$ since $A^* \in \A_M^k$. 
    We have by Lemma~\ref{lem:momuniform} that, with probability at least \( 1 - \delta \),
    \begin{align*}
        \sup_{A \in \A_{M}^k} \abs{\E l_{A}(X) - \MOM(l_{A})} 
        &\lesssim \E \sup_{A \in \A_{M}^k} \frac{1}{N} \sum_{i = 1}^{N} \eps_i l_{A}(X_i) + \sqrt{\sup_{A \in \A_{M}^k} \Var(l_{A}(X)) \frac{\log \frac{1}{\delta}}{N}},
    \end{align*}
    where \( \eps_1, \dots, \eps_N \) are independent Rademacher signs.
    Here, we have for each \( A \in \A_{M}^k \),
    \begin{align*}
        l_{A}(x)^2 & = \left(\norm{x - a_x}^{2} - \norm{x}^{2}\right)^2 = (\norm{x - a_x} - \norm{x})^2 (\norm{x - a_x} + \norm{x})^2 \\
        & \leq \norm{a_x}^2 (2 \norm{x} + \norm{a_x})^2,
    \end{align*}
    where \( a_x \in \mathrm{Arg}\min_{a \in A} \norm{x - a} \). Then, since \( \norm{a_x} \leq M \) for any \( x \), we have
    \[
        \Var(l_{A}(X)) \leq \E l_{A}(X)^2 \lesssim M^2 \left(M^2 + \E \norm{X}^{2}\right) \, .
    \]
    
    \textbf{Step 2.} 
    Note that $\hat{A}_{\delta, M}$ can consist of less than $k$ points. However, in this case we can always add the copies of some of them and identify $\hat{A}_{\delta, M}$ with $(a_1, \dots, a_k)$. 
    This does not change $l_{\hat{A}_{\delta, M}}$ and preserves the Voronoi partition of the space since the cells corresponding to the newly added points are empty. Finally, we estimate
    \begin{equation}
    \label{supofproc}
        \E \sup_{A \in \A_{M}^k} \frac{1}{N} \sum_{i = 1}^{N} \eps_i l_{A}(X_i).
    \end{equation}
    Consider the set $\F_{\A} = \left\{f_{A} \;:\; A \in \A_{M}^k\right\}$ of $\R^k$-valued functions such that for any $A = \{a_1, \ldots, a_k\},\ A \in \A_{M}^k$ we set
    \begin{equation}
    \label{fax}
        f_{A}(x) = \left(-2\langle x,a_1 \rangle + \norm{a_1}^2, \ldots, -2\langle x,a_k \rangle + \norm{a_k}^2\right).
    \end{equation}
    For $\cv \in \R^k$ let $\phi(\cv) = \min\limits_{i \leq k} c_i$. We obviously have $l_{A}(X) = \phi(f_{A}(X))$. Following the analysis of Section~3.2 in \citep{maurer2016vector} we have for any two $A$ and $B$ in $\A_M^k$,
    \[
        \abs*{\phi(f_{A}(X_i)) - \phi(f_{B}(X_i))} \le \norm*{f_A(X_i) - f_B(X_i)}_2.
    \]
    This allows us to use the $\ell_2$-contraction to upper bound \eqref{supofproc} with the quantity scaling linearly in $k$. To do so, we observe that Maurer's vector contraction inequality (Theorem~3 in \citep{maurer2016vector}) implies
    \[
        \E \sup_{A \in \A_{M}^k} \frac{1}{N} \sum_{i = 1}^{N} \eps_i l_{A}(X_i) 
        \le \frac{\sqrt{2}}{N} \left(2 \E \sup\limits_{A \in \A_{M}^k} \sum_{i, j = 1}^{N, k}\eps_{i, j}\langle X_i, a_j\rangle + \E\sup\limits_{A \in \A_{M}^k} \sum_{i, j = 1}^{N, k} \eps_{i, j} \norm{a_j}^2\right),
    \]
    where $\eps_{i, j}$, $i = 1, \ldots, N$, $j = 1, \ldots, k$, are independent Rademacher signs, and where $A = \{a_1, \ldots, a_k\}$. We further have by Khintchine's inequality,
    \begin{align*}
        \E \sup\limits_{A  \in \A_{M}^k} \sum_{i, j = 1}^{N, k} \eps_{i, j} \langle X_i, a_j\rangle &\le \sum\limits_{j = 1}^{k} \E \sup\limits_{A  \in \A_{M}^k} \left\langle \sum \limits_{i = 1}^{N} \eps_{i,j} X_i, a_j\right\rangle \le \sum\limits_{j = 1}^{k} \E \sup\limits_{A  \in \A_{M}^k} \norm*{\sum\limits_{i = 1}^{k} \eps_{i,j} X_i} \norm{a_j}
        \\
        &\le k M \max\limits_{j \leq k} \E \norm*{\sum\limits_{i = 1}^{N} \eps_{i,j} X_i} \le k M \max\limits_{j \leq k} \sqrt{\E \norm*{\sum\limits_{i = 1}^{N} \eps_{i,j} X_i}^2}
        \\
        &\le k M \sqrt{\sum\limits_{i = 1}^{N} \E \norm{X_i}^2},
    \end{align*}
    and also
    \[
        \E \sup\limits_{A \in \A_{M}^k} \sum_{i, j = 1}^{N, k} \eps_{i, j} \norm{a_j}^2 
        \le \sum\limits_{j = 1}^{k} \E \sup\limits_{A \in \A_{M}^k} \abs*{\sum_{i = 1}^{N}\eps_{i, j}} \norm{a_j}^2 
        \lesssim k M^2 \sqrt{N}.
    \]
    Finally, taking the expectation with respect to $X_1, \ldots, X_N$ and using Jensen's inequality we obtain an analog of~\eqref{boundedcase}. That is,
    \begin{equation}
    \label{lineark}
        \E \sup_{A \in \A_{M}^k} \frac{1}{N} \sum_{i = 1}^{N} \eps_i l_{A}(X_i) \lesssim \frac{k M \left(\sqrt{\E \norm{X}^2} + M\right)}{\sqrt{N}}.
    \end{equation}
    Combining the above bounds we prove the claim.
\end{proof}

It is by now well known that in our setup in the bounded case (e.g., when $\norm{X} \le T$ almost surely) the right dependence of the excess distortion on the number of clusters is $\sqrt{k}$ up to logarithmic factors \citep{fefferman2016testing, narayanan2010sample}. It is natural to ask if the same dependence is possible in our Theorem~\ref{thm:simpleupperm}. First, observe that in the unbounded case, there are some complications. In particular, our parameter $M = M(P, k)$ can also depend on $k$. This means that the overall dependence of the excess distortion on $k$ can be more complicated. Nevertheless, in the next section we show, among other things, that these improvements are possible and, in particular, the $k$-term is replaced by the $\sqrt{k}$-term. 

\section{Towards Better Bounds Based on \texorpdfstring{\(\pmin\)}{pmin}}
\label{sec:pminbound}

This section is devoted to our main results. We prove almost optimal non-asymptotic bounds for $k$-means. Recall that if $\E \norm{X}^2 < \infty$ we have for any optimal quantizer
\[
\pmin = \min_{a \in A^*} P(V_{a}) > 0,
\]
unless the support of \(P\) has less than \(k\) points. Notice that \( \pmin \) controls the magnitude of the largest vector in \(A^{*}\). Indeed, using the centroid condition, Jensen's inequality, and the Cauchy--Schwarz inequality we have for any $a \in A^*$,
\begin{equation}
\label{eq:boundsona}
    \norm{a} = \norm*{\E[X \vert \; X \in V_{a}]} 
    = \frac{\norm*{\E X \Ind[X \in V_a]}}{P(V_{a})} 
    \leq \frac{\E^{1/2} \norm{X}^{2}}{\sqrt{P(V_{a})}} 
    \le \frac{\E^{1/2} \norm{X}^{2}}{\sqrt{\pmin}} .
\end{equation}
This suggests that the mass of the lightest cluster of an optimal quantizer should affect the quality of any empirical quantizer. 

Let us return to Example~\ref{pminexample}. In this case we have $k = 2$, $M \le \sqrt{N}$, $\pmin = \frac{1}{N}$, $\E \norm{X}^2 \le 1$ and the bound \eqref{eq:boundsona} is tight. However, the bound of Theorem~\ref{thm:simpleupperm} is not tight anymore as it scales as $O(\sqrt{N})$. Indeed, Theorem~\ref{thm:simpleupperm} implies the bound 
$O\left(\frac{k \left(M^2 + M \sqrt{\E \norm{X}^2}\right)}{\sqrt{N}}\right)$ which is $O\left(\frac{k \E \norm{X}^2}{\pmin \sqrt{N}}\right)$ whenever \eqref{eq:boundsona} is tight.

The challenging part is to get the optimal dependence on $\pmin$ and $N$ in the excess distortion bound. In what follows, we show that the dependence $\Theta \left(\frac{1}{\sqrt{N \pmin}}\right)$ is achievable with respect to these parameters. The result of this form guarantees the consistency for sequences of distributions depending on $N$ as long as \(N \pmin \to \infty\) and the second moment is uniformly bounded. This extends the original asymptotic result of \cite{pollard1981strong} to the case where the distribution is allowed to change with $N$.

Suppose that we know the value of \(\pmin > 0\)  for at least one optimal quantizer.
Denote for short, \(P_N(V) = \frac{1}{N} \sum_{i = 1}^{N} \Ind[X_i \in V]\). Naturally, we want to find a solution \(\hat{A}\) such that the corresponding Voronoi cells are of measure at least \(\pmin > 0\) which translates into \(P_N(V_j) \ge \pmin/2\) due to concentration, provided that $N$ is large enough.
It implies that each cell corresponding to $\hat{A}$ contains enough sample points, which corresponds to the so-called \emph{constrained} $k$-means clustering. 
In $\R^d$ the algorithmic side of constrained clustering is well studied in the context of optimal transport and has numerous practical applications, see \citep*{ng2000note, cuturi2014fast, genevay2019differentiable} and references therein. 
We have additional motivation to introduce $\pmin$ since this quantity appears naturally in the condition implying the so-called \emph{fast rates} of the excess distortion in the bounded case \citep{levrard2015nonasymptotic}. 
Finally, recalling Example~\ref{pminexample} we know that in any reasonable clustering problem $\pmin \gg \frac{1}{N}$ which means that the optimal solution $A^*$ has enough observations in each cell. At the same time, we do not have such a natural preliminary guess on $M$.

As before, the number of blocks depends solely on the desired confidence level. 
Our main result is the following theorem.

\begin{Theorem}\label{thm:pmin}
    Fix $\delta \in (0, 1)$. Suppose, \(\min_{a \in A^{*}} P(V_a) \geq  \pmin \) for some optimal quantizer \(A^{*} \).
    There is an estimator $\hat{A}_{\delta, \pmin}$ that depends on $\pmin$ and $\delta$ such that, with probability at least \( 1 - \delta\),
    \[
    D(\hat{A}_{\delta, \pmin}) - D(A^*) \lesssim \E \norm{X - \mu}^{2} \left((\log N)^2 \sqrt{\frac{k}{N \pmin}} + \sqrt{\frac{\log \frac{1}{\delta}}{N \pmin}}\right) \, .
    \]
\end{Theorem}

Let us now present our estimator.

\begin{framed}
    \textbf{The estimator of Theorem~\ref{thm:pmin}}. We set 
    \begin{equation}\label{pmin_estimator}
        \hat{A}_{\delta, \pmin} = \argmin_{\substack{A \in \A^k \\ \min\limits_{a \in A} P_N(V_a) \geq \pmin / 2}} \MOM(l_{A}),
    \end{equation}
    with the number of blocks $\ell = 12 \left\lceil \log \tfrac{6}{\delta}\right\rceil + 1$.
\end{framed}

The idea behind this estimator is quite natural: we guarantee the robustness by using the MOM principle and by restricting our attention only to the cells containing enough points. As already mentioned, this is essentially a robust version of the constrained $k$-means quantizer introduced in \cite{ng2000note}.

We introduce several technical results that together lead us to the proof of Theorem~\ref{thm:pmin} at the end of this section. 
Since the estimator we consider is translation invariant, we can assume that \( \E X = 0 \) in the proof without loss of generality. As previously, our main tool is the concentration of MOM for a suitably chosen subset of \(\left\{l_{A} \;:\; A \in \A^k\right\}\). We show that the restriction \( P_{N}(V_j) \geq \pmin / 2 \) in \eqref{pmin_estimator} implies a convenient bound for the vectors in the resulting empirical quantizer. Let us define the following class of quantizers:
\begin{equation}\label{def:A_Mm}
    \A_{M,m}^k = \Bigl\{A \in \A^k \;:\; \min_{a \in A} \norm{a} \le m,\; \max_{a \in A} \norm{a} \le M\Bigr\}, \quad 0 < m \le M.
\end{equation}

The following lemma says that with high probability all the solutions corresponding to $\hat{A}_{\delta, \pmin}$ are bounded which is, of course, natural in view of the proof of Theorem~\ref{thm:simpleupperm}. However, the key technical observation is that we also need to control the smallest norm by saying that there is at least one center in $\hat{A}_{\delta, \pmin}$ which is relatively close to the actual expectation. In order to show this we do not have to use any uniform results that hold simultaneously for the entire class $\A^k$. Therefore, we have the following property.

\begin{Lemma}
\label{lem:optimizer_Mm}
    With probability at least $1 - e^{- \ell / 12}$, it holds that simultaneously for all $A \in \A^k$ such that $\MOM(l_A) \le 0$,
    \[
    \min_{a \in A} \norm{a} \le m = 4 \sqrt{2 \E \norm{X}^2}.
    \]
    Moreover, with probability at least $1 - \left(e^{- \ell / 12} + e^{- N \pmin / 12}\right)$, it holds that
    \[
    \hat{A}_{\delta, \pmin} \in \A^k_{M, m}
    \quad\text{for } m \text{ defined above and } M = 10 \sqrt{\frac{\E \norm{X}^2}{\pmin}}.
    \]
\end{Lemma}

\begin{Remark}\label{rem:simple_bound}
    Note that the first statement of the above lemma gives us a prior bound on the excess distortion $D(\hat{A}_{\delta, \pmin}) - D(A^*)$. Indeed, since $\ell \ge 12 \log \tfrac{1}{\delta}$, with probability at least $1 - \delta$, we have $\min_{a \in \hat{A}_{\delta, \pmin}} \norm{a} \le m$, thus
    \begin{equation}\label{eq:simple_bound}
        D(\hat{A}_{\delta, \pmin}) - D(A^*) \le D(\hat{A}_{\delta, \pmin}) = \E \min_{a \in \hat{A}_{\delta, \pmin}} \norm{a - X}^2 \le \E (m + \norm{X})^2 \lesssim \E \norm{X}^2.
    \end{equation}
\end{Remark}

Before going to the proof of Lemma~\ref{lem:optimizer_Mm}, let us state the following trivial result on empirical quantiles. We postpone its proof to the appendix.

\begin{Lemma}
\label{lem:quant_ratio}
    Let $\xi_1, \dots, \xi_\ell$ be i.i.d.\ random values such that $\E \xi < \infty$ and $\xi \ge 0$ almost surely. 
    Then for any $0 < \alpha < 1$ we have
    \[
    \P\left(\mathrm{Quant}_{1 - \alpha}(\xi_1, \dots, \xi_\ell) \ge \frac{2}{\alpha} \E \xi\right) \le \exp\left(- \frac{\alpha \ell}{6}\right).
    \]
\end{Lemma}

\begin{proof}[Proof of Lemma~\ref{lem:optimizer_Mm}]
    \textbf{Step 1.} 
    First, let us prove the bound on the minimal norm.
    Consider $A \in \A^k$ such that $\min_{a \in A} \norm{a} \ge m$. Then for any $x \in B_{m / 2}$ (recall that $B_{m/2}$ is a ball of radius $m/2$ centered at the origin) it holds that $\min_{a \in A} \norm{a - x} \ge \frac{m}{2}$, thus for all $x \in E$,
    \begin{align*}
        l_A(x) = \min_{a \in A} \norm{a - x}^2 - \norm{x}^2 
        &\ge \frac{m^2}{4} \Ind\left[\norm{x} \le \frac{m}{2}\right] - \norm{x}^2
        \\
        &= \frac{m^2}{4} - \left(\frac{m^2}{4} \Ind\left[\norm{x} > \frac{m}{2}\right] + \norm{x}^2\right),
    \end{align*}
    and hence 
    \[
    \MOM(l_A) \ge \frac{m^2}{4} - \MOM\left(\frac{m^2}{4} \Ind\left[\norm{X} > \frac{m}{2}\right] + \norm{X}^2\right).
    \]
    According to Lemma~\ref{lem:quant_ratio}, with probability at least $1 - e^{- \ell / 12}$, it holds that
    \begin{align*}
        \MOM\left(\frac{m^2}{4} \Ind\left[\norm{X} > \frac{m}{2}\right] + \norm{X}^2\right) 
        &\le 4 \E \left(\frac{m^2}{4} \Ind\left[\norm{X} > \frac{m}{2}\right] + \norm{X}^2\right) \\
        &= 4 \left(\frac{m}{2}\right)^{2} \P\left(\norm{X} > \frac{m}{2} \right) + 4 \E \norm{X}^{2} \\
        &< 8 \E \norm{X}^2,
    \end{align*}
    where in the last inequality we apply strict Chebyshev's inequality (assuming \(\E \norm{X}^2 > 0\)). Thus, simultaneously for all $A \in \A^k$ satisfying $\min_{a \in A} \norm{a} \ge m$ we have
    \[
    \MOM(l_A) > \frac{m^2}{4} - 8 \E \norm{X}^2 = 0.
    \]
    In particular, since $\{0\}$ is one of the potential candidates for $\hat{A}_{\delta, \pmin}$ and $\MOM(l_{\{0\}}) = 0$, we have
    \[
    \min_{a \in \hat{A}_{\delta, \pmin}} \norm{a} < m.
    \]
    
    \textbf{Step 2.}
    Now consider $A \in \A^k$ such that there is $b \in A$ with $\norm{b} \le m$.
    It is easy to see that for any $a \in A$, $x \in V_a$ implies $\norm{a - x} \le \norm{b - x}$, thus
    \begin{equation}\label{eq:norm_x}
        \norm{x} \ge \frac{\norm{a} - \norm{b}}{2}.
    \end{equation}
    Assume $\norm{a} > M$ for some $a \in A$, then $\norm{x} > \frac{M - m}{2}$ for any $x \in V_a$. Hence,
    \[
    P_N(V_a) \le P_N\left(\norm{X} > (M - m)/2\right).
    \]
    At the same time, by strict Chebyshev's inequality we have (once $\E \norm{X}^2 > 0$) that
    \[
    P\left(\norm{X} > (M - m)/2\right) < \frac{4 \E \norm{X}^2}{(M - m)^2} \le \frac{\pmin}{4}.
    \]
    
    Now Chernoff's bound for Bernoulli random variables (p.~48, \cite{boucheron2013concentration}) yields that, with probability at least \(1 - e^{- N \pmin / 12}\), 
    \[
        P_N\left(\norm{X} > (M - m)/2\right) 
        \le P\left(\norm{X} > (M - m)/2\right) + \frac{\pmin}{4}
        < \frac{\pmin}{2}.
    \]
    This implies \(P_N(V_a) < \frac{\pmin}{2}\),
    what means that none of such $A$ can be chosen by our estimator. 
    By the union bound, we finally get that $\hat{A}_{\delta, \pmin} \in \A^k_{M, m}$, with probability at least \(1 - \left(e^{- \ell / 12} + e^{- N \pmin / 12}\right)\).
\end{proof}

The next step is to provide a uniform concentration of MOM over a class of quantizers \(\A^k_{M, m}\). First, we estimate the $L_2$-diameter and covering numbers of the functional class corresponding to $\A_{M, m}^k$: 
\begin{equation}\label{def:F_Mm}
    \F_{M, m}^k = \left\{l_A \;:\; A \in \A_{M, m}^k\right\}.
\end{equation}

\begin{Lemma}
\label{lem:A_radius}
    For any distribution $P$ and any set $A \in \A^k_{M, m}$ it holds that
    \begin{equation}\label{eq:A_radius}
        \sum_{a \in A} \norm{a}^2 P(V_a) \le 2 m^2 + 8 \E \norm{X}^2,
    \end{equation}
    and
    \begin{equation}\label{eq:var_Mm}
        \E l_A^2(X) \le 4 M^2 \left(m^2 + 6 \E \norm{X}^2\right).
    \end{equation}
\end{Lemma}

\begin{proof}
    Fix $A \in \A^k$ and let $b \in A$ be such that $\norm{b} \le m$. 
    Then for any $a \in A$ and $x \in V_a$ it holds from~\eqref{eq:norm_x} that \(\norm{a} \le \norm{b} + 2 \norm{x} \le m + 2 \norm{x}\).
    Therefore,
    \[
    \sum_{a \in A} P(V_a) \norm{a}^2 = \E \sum_{a \in A} \Ind[X \in V_a] \norm{a}^2 
    \le \E \left(m + 2 \norm{X}\right)^2 \le 2 m^2 + 8 \E \norm{X}^2.
    \]
    Further, we easily have using \eqref{eq:A_radius}
    \begin{align*}
        \E l_A^2(X) 
        &\le \sum_{a \in A} P(V_a) \E\left[\left(\norm{a}^2 + 2 \norm{a} \cdot \norm{X}\right)^2 \middle| X \in V_a\right] \\
        & \le 2 M^2 \left(\sum_{a \in A} P(V_a) \norm{a}^2 + 4 \E \norm{X}^2\right) \\
        & \le 4 M^2 \left(m^2 + 6 \E \norm{X}^2\right).
    \end{align*}
\end{proof}

The next technical lemma is one of our main contributions which can be of independent interest. It states the upper bounds on $\log \mathcal{N}_{2}\left(\F_{M, m}^k, t, P_N\right)$ for general separable Hilbert spaces as well as for $\R^d$. The question on sharp bounds on covering numbers for the classes of functions indexed by $\A^k$ appeared naturally in the analysis of $k$-means clustering in the uniformly bounded case. The way to do it is to estimate the so-called fat-shattering dimension \citep*{narayanan2010sample, fefferman2016testing} or to decompose the covering numbers as a product of $k$ covering numbers of some simpler classes indexed by $\A$ as in \citep*{brownlees2015empirical, foster2019ell_}. Furthermore, in the special case of $\R^d$, the analysis can be done via the computation of Pollard's pseudodimension \citep{bachem2017uniform}.
Unfortunately, it seems that these approaches are better tuned to the analysis of uniformly bounded distributions or to the finite dimensional case. Our approach is based on direct computations of these covering numbers via the Johnson--Lindenstrauss lemma \citep*{johnson1984extensions}, and in $\R^d$ our analysis removes the unnecessary logarithmic factors appearing in some previous works in the bounded case.

\begin{Lemma}
\label{lem:cov_number}
    For any \(0 < t < \diam_{2}\left(\F_{M, m}^k, P_N\right)\) it holds that
    \begin{equation}
    \label{eq:cov_number}
        \log \mathcal{N}_{2}\left(\F_{M, m}^k, t, P_N\right)
        \lesssim \frac{k M^2 \left(m^2 + P_N \norm{X}^2\right) \log N}{t^2} \log \frac{M \left(m + \sqrt{P_N \norm{X}^2}\right)}{t}.
    \end{equation}
    Moreover, if $E = \R^d$, then
    \begin{equation}
    \label{eq:cov_number_Rd}
        \log \mathcal{N}_{2}\left(\F_{M, m}^k, t, P_N\right)
        \lesssim k d \log \frac{M \left(m + \sqrt{P_N \norm{X}^2}\right)}{t}.
    \end{equation}
\end{Lemma}

The proof of this fact is deferred to the appendix. With this result in mind we are ready to show the following uniform bound.
\begin{Lemma}
\label{lem:concentration_Mm}
    Fix the number of blocks $\ell$ and assume that $\ell$ divides $N$.
    Then for any fixed $\alpha \in (0, 1)$ with \( \ell \alpha \) being non-integer, with probability at least \(1 - e^{- \alpha^2 \ell / 2}\),
    \begin{equation}
    \label{eq:concentration_Mm}
        \sup_{A \in \A_{M, m}^k} \left(\E l_{A}(X) - \QOM_{\alpha}(l_{A})\right)
        \lesssim M \left(m + \sqrt{\E \norm{X}^2}\right) \left(\frac{(\log N)^2}{\alpha} \sqrt{\frac{k}{N}} + \sqrt{\frac{\ell}{\alpha N}}\right),
    \end{equation}
    as well as, with probability at least \(1 - e^{- (1 - \alpha)^2 \ell / 2}\),
    \begin{align}
    \label{eq:concentration_Mm_uppertail}
        &\sup_{A \in \A_{M, m}^k} \left(\QOM_{\alpha}(l_{A}) - \E l_{A}(X) \right)
        \\
        &\quad\quad\lesssim M \left(m + \sqrt{\E \norm{X}^2}\right) \left(\frac{(\log N)^2}{1 - \alpha} \sqrt{\frac{k}{N}} + \sqrt{\frac{\ell}{(1 - \alpha) N}}\right) \nonumber.
    \end{align}
\end{Lemma}

\begin{proof}
    By Lemma~\ref{lem:momuniform}, with probability at least \(1 - e^{- \alpha^2 \ell / 2}\), it holds that
    \begin{align*}
        &\sup_{A \in \A_{M, m}^k} \left(\E l_{A}(X) - \QOM_{\alpha}(l_{A})\right)
        \\
        &\quad\quad\lesssim \E \sup_{A \in \A_{M, m}^k} \left(\frac{1}{\alpha N} \sum_{i = 1}^{N} \eps_i l_A(X_i)\right) + \sqrt{\sup_{A \in \A_{M, m}^k} \Var(l_A(X)) \frac{\ell}{\alpha N}}.
    \end{align*}
    
    Now we are going to bound the first term of the right-hand side for a fixed sample $X_1, \dots, X_N$ using Dudley's chaining argument. 
    It follows from~\eqref{eq:var_Mm} applied to the empirical distribution $P_N$ that
    \[
    \diam_2\left(\F_{M, m}^k, P_N\right) \le 10 M \sigma_N, \quad\text{where}\quad \sigma_N = m + \sqrt{P_N \norm{X}^2},
    \]
    thus, Dudley's chaining argument (e.g., Lemma~A.3 in \citep{srebro2010smoothness}) together with Lemma~\ref{lem:cov_number} ensure the following bound on the Rademacher averages of $l_A$ for any $\beta > 0$,
    \begin{align*}
        \E_{\eps} \sup_{A \in \A_{M, m}^k} \left(\frac{1}{N} \sum_{i = 1}^{N} \eps_i l_A(X_i)\right) 
        &\lesssim \beta + \frac{1}{\sqrt{N}} \int_\beta^{\diam_2\left(\F_{M, m}^k, P_N\right)} \sqrt{\log\mathcal{N}_2 \left(\F_{M, m}^k, t, P_N\right)} d t \\
        &\lesssim \beta + \frac{1}{\sqrt{N}} \int_\beta^{10 M \sigma_N} \frac{M \sigma_N}{t} \sqrt{k \log N \log \frac{M \sigma_N}{t}} d t \\
        &\lesssim \beta + M \sigma_N \sqrt{\frac{k \log N}{N}} \log^{3/2}\left(\frac{M \sigma_N}{\beta}\right)
    \end{align*}
    Further, choosing \(\beta = M \sigma_N \sqrt{\frac{k}{N}}\) we have
    \[
        \E_{\eps} \sup_{A \in \A_{M, m}^k} \left(\frac{1}{N} \sum_{i = 1}^{N} \eps_i l_A(X_i)\right) 
        \lesssim M \sigma_N \sqrt{\frac{k}{N} \log N \log^{3}\frac{N}{k}}
        \le M \sigma_N (\log N)^2 \sqrt{\frac{k}{N}}.
    \]
    Finally,~\eqref{eq:var_Mm} implies
    \[
    \Var(l_A(X)) \le \E l_A^2(X) \lesssim M^2 \left(m^2 + \E \norm{X}^2\right),
    \]
    thus we conclude that, with probability at least \(1 - e^{- \alpha^2 \ell / 2}\),
    \begin{align*}
        \sup_{A \in \A_{M, m}^k} \left(\E l_{A}(X) - \QOM_{\alpha}(l_{A})\right)
        & \lesssim \frac{\E M \sigma_N (\log N)^2}{\alpha} \sqrt{\frac{k}{N}} + \sqrt{\sup_{A \in \A_{M, m}^k} \Var(l_A(X)) \frac{\ell}{\alpha N}} \\
        & \lesssim M \left(m + \sqrt{\E \norm{X}^2}\right) \left(\frac{(\log N)^2}{\alpha} \sqrt{\frac{k}{N}} + \sqrt{\frac{\ell}{\alpha N}}\right).
    \end{align*}
Inequality \eqref{eq:concentration_Mm_uppertail} can be similarly derived from \eqref{absvalueequation_second}.
\end{proof}

We are now ready to prove Theorem~\ref{thm:pmin}.

\begin{proof}[Proof of Theorem~\ref{thm:pmin}]
    In order to finish the proof we need to combine several results.
    Let us fix some optimal quantizer \(A^*\), satisfying \(P(V_a) \ge \pmin\) for all \(a \in A^*\).
    We derive the bound on the union of the events below:
    \begin{itemize}
        \item by Chernoff's and the union bounds, it holds with probability at least \(1 - k e^{-N \pmin / 8}\) that for any \(a \in A^*\)
        \[
        P_N(V_a) \ge \frac{P(V_a)}{2} \ge \frac{\pmin}{2},
        \]
        hence \(A^*\) is in the set of possible solutions on this event;
        
        \item by Lemma~\ref{lem:optimizer_Mm}, with probability at least \( 1 - e^{- \ell / 12} - e^{- N \pmin / 12}\), we have \(\hat{A}_{\delta, \pmin} \in \A^k_{M, m}\) with 
        \(M = 10 \sqrt{\frac{\E \norm{X}^{2}}{\pmin}}\) and \(m = 4 \sqrt{2 \E \norm{X}^{2}}\).
        In addition, a similar property can be derived for \(A^*\). Indeed, if \(\min_{a \in A^*} \norm{a} > m\), then
        \begin{align*}
            \E l_{A^*}(X) &\ge \frac{m^2}{4} P\left(\norm{X} \le \frac{m}{2}\right) - \E \norm{X}^2
            \\
            &\ge 8 \E \norm{X}^{2} \left(1 - \frac{1}{8}\right) - \E \norm{X}^{2} 
            > 0 = \E l_{\{0\}}(X).
        \end{align*}
        This contradicts the optimality of \(A^*\).
        Now assume \(\min_{a \in A^*} \norm{a} \le m\), but \(\norm{a} > M\) for some \(a \in A^*\). Arguing as in the proof of Lemma~\ref{lem:optimizer_Mm}, we obtain
        \[
        P(V_a) \le P\left(\norm{X} > \frac{M - m}{2}\right) < \frac{4 \E \norm{X}^2}{(M - m)^2} \le \frac{\pmin}{4}.
        \]
        This contradicts the lower bound \(P(V_a) \ge \pmin\);
        
        \item by Lemma~\ref{lem:concentration_Mm}, taking $M$ and $m$ as above, 
        we have that, with probability at least \(1 - 2 e^{-\ell/8}\),
        \[
            \sup_{A \in A_{M, m}} \abs{\E l_{A} - \MOM(l_{A})} 
            \lesssim \E \norm{X}^2 \left((\log N)^2 \sqrt{\frac{k}{N \pmin}} + \sqrt{\frac{\ell}{N \pmin}}\right) \, .
        \]
    \end{itemize}
    All three assertions take place with probability at least \( 1 - 3 e^{-\ell / 12} - (k + 1) e^{-N \pmin / 12}\). Suppose for a moment that \( N \pmin \geq 12 \log \tfrac{2 (k + 1)}{\delta} \). Then, additionally, due to the choice \(\ell = 12 \left\lceil \log \tfrac{6}{\delta} \right\rceil + 1 \), we have that the total probability is at least \( 1 - \delta \).
    Since we know that on this event \(\hat{A}_{\delta, \pmin}, A^* \in A_{M, m} \) and \( \MOM\left(l_{\hat{A}_{\delta, \pmin}}\right) \leq \MOM(l_{A^*}) \), we have
    \begin{align*}
        D(\hat{A}_{\delta, \pmin}) - D(A^*) & = \E l_{\hat{A}_{\delta, \pmin}} - \E l_{A^{*}}\\ 
        & \le \E l_{\hat{A}_{\delta, \pmin}} - \MOM\left(l_{\hat{A}_{\delta, \pmin}}\right) - \E l_{A^*} + \MOM(l_{A^*})\\ 
        & \lesssim \E \norm{X}^2 \left((\log N)^2 \sqrt{\frac{k}{N \pmin}} + \sqrt{\frac{\log \frac{1}{\delta}}{N \pmin}}\right) \, .
    \end{align*}
    
    Finally, consider the case \( N \pmin < 12 \log \tfrac{2 (k + 1)}{\delta} \). According to~\eqref{eq:simple_bound} one has, with probability at least $1 - \delta$,
    \[
    D(\hat{A}_{\delta, \pmin}) - D(A^*) \lesssim \E \norm{X}^2 
    \lesssim \E \norm{X}^2 \left((\log N)^2 \sqrt{\frac{k}{N \pmin}} + \sqrt{\frac{\log \frac{1}{\delta}}{N \pmin}}\right) \, .
    \]
    The claim follows.
\end{proof}

Finally, we present an analog of Theorem~\ref{thm:pmin} in $\R^d$. We are able to completely remove the $(\log N)$-factor by making $d$ appear in the bound. First, we need the following simple result.

\begin{Corollary}
\label{lem:concentration_Mm_Rd}
    For any $\alpha \in (0, 1)$ with \(\ell \alpha\) being non-integer, we have, with probability at least \(1 - e^{- \alpha^2 \ell / 2}\),
    \[
        \sup_{A \in \A_{M, m}^k} \left(\E l_{A}(X) - \QOM_{\alpha}(l_{A})\right)
        \lesssim M \left(m + \sqrt{\E \norm{X}^2}\right) \left(\frac{1}{\alpha} \sqrt{\frac{k d}{N}} + \sqrt{\frac{\ell}{\alpha N}}\right).
    \]
    as well as, with probability at least \(1 - e^{- (1 - \alpha)^2 \ell / 2}\),
    \[
        \sup_{A \in \A_{M, m}^k} \left(\QOM_{\alpha}(l_{A}) - \E l_{A}(X) \right)
        \lesssim M \left(m + \sqrt{\E \norm{X}^2}\right) \left(\frac{1}{1 - \alpha} \sqrt{\frac{k d}{N}} + \sqrt{\frac{\ell}{(1 - \alpha) N}}\right).
    \]
\end{Corollary}

\begin{proof} 
    Using the Dudley integral argument again we have
    \begin{align*}
        \E_{\eps} \sup_{A \in \A_{M, m}^k} \left(\frac{1}{N} \sum_{i = 1}^{N} \eps_i l_A(X_i)\right) 
        &\lesssim \frac{1}{\sqrt{N}} \int_0^{\diam_2\left(\F_{M, m}^k, P_N\right)} \sqrt{\log\mathcal{N}_2 \left(\F_{M, m}^k, t, P_N\right)} d t \\
        &\lesssim \frac{1}{\sqrt{N}} \int_0^{10 M \sigma_N} \sqrt{k d \log\left(\frac{M \sigma_N}{t}\right)} d t \\
        &\lesssim M \sigma_N \sqrt{\frac{k d}{N}}.
    \end{align*}
    The rest of the proof is exactly the same as for Lemma~\ref{lem:concentration_Mm}.
\end{proof}

With this result in mind, we can immediately prove our second main result.

\begin{Theorem}
\label{thm:pmin_Rd}
    Consider the case of $\R^d$ with the Euclidean distance. Fix \(\delta \in (0, 1)\). Suppose, \(\min_{a \in A^{*}} P(V_a) \geq  \pmin \) for some optimal quantizer \(A^{*} \).
    The same estimator $\hat{A}_{\delta, \pmin}$ satisfies, with probability at least \( 1 - \delta\),
    \[
        D(\hat{A}_{\delta, \pmin}) - D(A^*) 
        \lesssim \E \norm{X - \mu}^{2} \sqrt{\frac{k d + \log \frac{1}{\delta}}{N \pmin}} \, .
    \]
    Moreover, if \(N \pmin \gtrsim d\log N + \log \frac{1}{\delta}\), then 
    \[
        D(\hat{A}_{\delta, \pmin}) - D(A^*) 
        \lesssim \sqrt{\Tr(\Sigma) \left(\lambda_{\max}(\Sigma) + \pmin \Tr(\Sigma)\right)} \sqrt{\frac{k d + \log \frac{1}{\delta}}{N \pmin}} \, ,
    \]
    where $\Sigma$ is the covariance matrix of $X$ and $\lambda_{\max}(\Sigma)$ is its largest eigenvalue.
\end{Theorem}

\begin{proof}
    The proof of the first statement repeats the same lines of the proof of Theorem~\ref{thm:pmin}. The only difference is that the bound of Lemma~\ref{lem:concentration_Mm} is replaced by the bound of Corollary~\ref{lem:concentration_Mm_Rd}.
    
    We proceed with the proof of the second inequality. In what follows, we only emphasize the differences with the proof of Theorem~\ref{thm:pmin}. 
    Our idea is to use the Vapnik--Chervonenkis type argument to provide a slightly sharper upper bound on $M$.
    Recall that according to Lemma~\ref{lem:optimizer_Mm}, with probability at least \( 1 - e^{-\ell/12}\), it holds that 
    \[
    \min_{a \in \hat{A}_{\delta, \pmin}} \norm{a} \le m = 4 \sqrt{2 \E \norm{X}^{2}}
    \] 
    and, with total probability $1 - \left(e^{- \ell / 12} + e^{- N \pmin / 12}\right)$, we have 
    \[
    \max_{a \in \hat{A}_{\delta, \pmin}} \norm{a} \le 10 \sqrt{\frac{\E \norm{X}^2}{\pmin}}.
    \]
    Now, we are going to show that, with probability $1 - \left(e^{- \ell / 12} + 4 e^{- N \pmin / 100}\right)$, \(\hat{A}_{\delta, \pmin} \in \A^k_{M, m}\), where \(M = 2 m \lor 8 \sqrt{\frac{\lambda_{\max}(\Sigma)}{\pmin}}\), provided that $N \pmin \gtrsim d \log N$. 
    Consider any $A \in \A^k$ such that there is $b \in A$ with $\norm{b} \le m$. Notice that $x \in V_a$ for $a \in A$ implies
    \begin{equation}
    \label{eq:condition}
        2 \langle x, a - b \rangle \ge \norm{a}^2 - \norm{b}^2.
    \end{equation}
    Using \(\E\left(\langle X, a - b \rangle\right)^2 \le \lambda_{\max}(\Sigma) \norm{a - b}^2\) and assuming that $\norm{a} > M$, we have by Chebyshev's inequality
    \begin{equation}
    \label{eq:pva}
        P\left(2 \langle X, a - b \rangle \ge \norm{a}^2 - \norm{b}^2\right) 
        \le \frac{4 \lambda_{\max}(\Sigma) \norm{a - b}^2}{(\norm{a}^2 - \norm{b}^2)^2} < \frac{16 \lambda_{\max}(\Sigma)}{M^2} \le \frac{\pmin}{4},
    \end{equation}
    where we additionally used
    \[
    \norm{a}^2 - \norm{b}^2 = \left(\norm{a} - \norm{b}\right) \left(\norm{a} + \norm{b}\right) > \frac{M \norm{a - b}}{2}.
    \]
    Observe that \eqref{eq:condition} implies
    \[
    P_N(V_a) \le P_N\left(2 \langle X, a - b \rangle \ge \norm{a}^2 - \norm{b}^2\right).
    \]
    Applying Theorem~5.1 in \citep*{boucheron2005theory} to the class induced by half-spaces \( \FF = \left\{ \Ind[\langle X, u \rangle \ge t] \;:\; u \in \R^{d}, t \in \R\right\}\), we get that simultaneously for all \(u \in \R^{d} \) and \( t \in \R\), with probability at least \(1 - 4 e^{- N \pmin / 100}\),
    \begin{align*}
        &P_N(\langle X, u \rangle \ge t) - P(\langle X, u \rangle \ge t) 
        \\
        &\quad\quad\le 2 \sqrt{P(\langle X, u \rangle \ge t) \frac{\log\mathbb{S}_{2N}(\FF) + N \pmin / 100}{N}} + 4 \frac{\log\mathbb{S}_{2N}(\FF) + N \pmin / 100}{N},
    \end{align*}
    where \( \mathbb{S}_{N}(\FF) \) denotes the shatter-coefficient of the class \( \FF \) (see the definition in \cite{boucheron2005theory}). By the \citeauthor{Vapnik74} lemma and since the VC-dimension of \( \FF \) is known to be $d + 1$, we have \( \log \mathbb{S}_{2N}(\FF) \lesssim d \log N\). Due to \eqref{eq:pva} and provided that \( N \pmin \ge 100 \log \mathbb{S}_{2N}(\FF)\), with probability at least \( 1 - 4 e^{- N \pmin / 100}\), we have 
    \[
    P_N\left(2 \langle X, a - b \rangle \ge \norm{a}^2 - \norm{b}^2\right) < \frac{\pmin}{2}.
    \]
    Therefore, we obtain a contradiction with the definition of $\hat{A}_{\delta, \pmin}$ and prove that, with probability at least \(1 - \left(e^{- \ell / 12} + 4 e^{- N \pmin / 100}\right)\), we have the inclusion \(\hat{A}_{\delta, \pmin} \in \A^k_{M, m}\). Similarly, \(A^* \in \A^k_{M, m}\). 
    Since \(N \gtrsim \frac{1}{\pmin} \ge k\), we have \( N \pmin \gtrsim d \log N + \log \frac{1}{\delta} \gtrsim \log \frac{k}{\delta}\).
    Then the remainder of the proof repeats the lines of the proof of Theorem~\ref{thm:pmin} in the regime \(N \pmin \gtrsim \log \tfrac {k}{\delta}\) with Lemma~\ref{lem:concentration_Mm} replaced by Corollary~\ref{lem:concentration_Mm_Rd}.
    The claim follows.
\end{proof}

Finally, we compare the results obtained in this section with the results from the previous section, were we study quantizers with known magnitude \( M \). First, we notice that the leading term \(M^2 + M \sqrt{\E \norm{X}^2}\) in Theorem~\ref{thm:simpleupperm} is replaced in~\eqref{eq:concentration_Mm} by a much better term \(M m + M \sqrt{\E \norm{X}^2}\). Indeed, in view of Lemma~\ref{lem:optimizer_Mm}, we are interested in the regime \(m \lesssim \sqrt{\E \norm{X}^{2}}\) and \( M \lesssim \sqrt{{\E \norm{X}^{2}}/{\pmin}} \). Therefore, the new bound allows us to obtain a better dependence on $\pmin$ since we may otherwise get an additional factor $\frac{1}{\sqrt{\pmin}}$.

Second, similar results can be obtained in the case where $M$ is known instead of $\pmin$, i.e., in the setting of Section~\ref{sec:warmup}. Namely, applying Lemma~\ref{lem:concentration_Mm} and the first part of Lemma~\ref{lem:optimizer_Mm} to the estimator of Theorem~\ref{thm:simpleupperm} we conclude that (see also the details of the proof of Theorem~\ref{thm:general_rate} below), with probability at least \(1 - \delta\), it holds that
\[
    D(\hat{A}_{\delta, M}) - D(A^*) \lesssim M \sqrt{\E \norm{X}^{2}} \left((\log N)^2 \sqrt{\frac{k}{N}} + \sqrt{\frac{\log \frac{1}{\delta}}{N}}\right) \, .
\]
In the case of $\R^d$, we obtain respectively that, with probability at least \( 1 - \delta\), it holds that
\[
D(\hat{A}_{\delta, M}) - D(A^*) \lesssim M \sqrt{\E \norm{X}^{2}} \left(\sqrt{\frac{k d}{N}} + \sqrt{\frac{\log \frac{1}{\delta}}{N}}\right) \, .
\]


\subsection{A lower bound with \texorpdfstring{\(\pmin\)}{pmin}}
\label{sec:lowerbound}

Here we study the question of the optimality of Theorem~\ref{thm:pmin} and Theorem~\ref{thm:pmin_Rd}. The lower bounds for the excess distortion appeared first in \citep{bartlett1998minimax} for the bounded case ($\norm{X} \le 1$ almost surely), where they showed a lower bound of order \( \Omega\left(\sqrt{\frac{k^{1 - 4 / d}}{N}}\right) \). Furthermore, \cite{linder2002learning} recovers this bound for constant \( d \) and \( k \geq 3 \), while \cite{antos2005improved} shows the same bound for \( k = 2 \) but only for empirically optimal quantizers. 
Below, we focus on how the mass of the lightest cluster affects the excess distortion in the unbounded case. We extend the construction of \cite{linder2002learning} to derive a bound that confirms that the dependence on \( \pmin \), $N$, and $\E \norm{X - \mu}^2$ in Theorem~\ref{thm:pmin_Rd} is sharp in some cases. The optimal dependence on the remaining parameters remains open.

Fix \( k = 4 \) and \( d = 1 \). Consider a class of probability measures on \( \R\), 
\[
    \mathcal{P}(\pmin, \sigma) = \left\{ P \;:\; \E X^{2} \leq \sigma^{2} \ \text{and}\; \exists A^* \in \argmin D(A, P) \ \text{s.t.}\ \min_{a \in A^*} P(V_{a}) \geq \pmin \right\},
\]
i.e., the probability measures that have an optimal quantizer based on \(k \) points such that the probability of \( X \) falling into each Voronoi cell under \( P \) is at least \( \pmin \). Theorem~\ref{thm:pmin_Rd} implies that for \( k = 4 \) and \( d = 1 \) there is an estimator \( \hat{A}_N \) based on the i.i.d.\ sample \( X_1, \dots, X_N \), such that for any \( P \in \mathcal{P}(\pmin, \sigma) \), we have with probability at least (say) $0.99$,
\[
    D(\hat{A}_N, P) - D(A^{*}, P) \lesssim \sigma^2 \sqrt{\frac{1}{N \pmin}},
\]
where the probability of the event is measured with respect to the joint distribution \( \P = P^{\otimes N} \).
The following result shows that when $d$, $k$, and $\delta$ are constants, the result is sharp up to a constant factor.

\begin{Theorem}\label{pmin_lowerbound}
    Under the notation introduced above let \( \sigma > 0 \), \( \pmin \leq {1}/{10}\). Let also \(N \pmin > 1/8\). Then, for any empirically designed quantizer \( \hat{A}_N \) there is a distribution \( P \in \mathcal{P}(\pmin, \sigma) \), such that, with probability at least $\frac{1}{4}$,
    \[
        D(\hat{A}_N, P) - D(A^{*}, P) \geq \frac{\sigma^{2}}{60}\sqrt{\frac{1}{N \pmin}}.
    \]
\end{Theorem}

Let us first present a heuristic argument showing the validity of Theorem~\ref{pmin_lowerbound}. Slightly abusing the notation for \( p \in (0, 1/2) \), \(\delta \in (-1/2, 1/2) \) consider a distribution \( P_{p, \delta} \) supported on five points 
\begin{align*}
    P_{p, \delta}(X = - \tfrac{1}{2} p^{-1/2}) &= P_{p, \delta}(X = - p^{-1/2}) = \frac{p(1 - \delta)}{4},
    \qquad
    P_{p, \delta}(X = 0)= 1 - p, \\
    P_{p, \delta}(X = \tfrac{1}{2} p^{-1/2}) &= P_{p, \delta}(X = p^{-1/2}) = \frac{p(1 + \delta)}{4}.
\end{align*}
We have \( \E X^{2} = 5/8  \). Obviously, we can rescale these values, so it is enough to consider the case \( \sigma^{2} = 5/8\). It is easy to see that for \( \delta > 0 \) the optimal quantizer is \( A^{*} = (0, \tfrac{1}{2} p^{-1/2}, p^{-1/2}, -\tfrac{3}{4} p^{-1/2}) \) with the distortion
\[
    D(A^*, P_{p, \delta}) = \frac{p(1 - \delta)}{2} \left( \frac{p^{-1/2}}{4}\right)^{2}   = \frac{1 - \delta}{32} .
\]
For \( \delta = \frac{1}{\sqrt{Np}} \), the number of points on the negative side is greater, with constant probability (see p. 27 in \citep{linder2002learning}). In such a case, the empirically optimal quantizer must be \( \hat{A} = (0, - p^{-1/2}, -\tfrac{1}{2} p^{-1/2}, (\frac{1}{2} + a)p^{-1/2}) \), where \( a \) is some value between \( 0\) and \( \tfrac{1}{2} \). Thus, the distortion of any empirically optimal quantizer is at least
\[
    D(\hat{A}, P_{p, \delta}) \geq \frac{p(1 + \delta)}{2} \left(\frac{p^{-1/2}}{4}\right)^{2}   = \frac{1 + \delta}{32},
\]
which implies
\[
    D(\hat{A}, P_{p, \delta}) - D(A^{*}, P_{p, \delta}) \geq \frac{\delta}{16} = \frac{1}{16} \frac{1}{\sqrt{Np}} \,.
\]
However, this only touches the empirically optimal quantizer. The proof of the lower bound relies on a standard reduction to hypothesis testing.

\begin{proof}[Proof of Theorem~\ref{pmin_lowerbound}]
    As pointed out above, we can fix \( \sigma^{2} = 5/8 \) without loss of generality. Set \( p = 4 \pmin \leq 2/5 \). Then we have that \( P_{p, \delta}, P_{p, -\delta} \in \mathcal{P}( \pmin, 5/8) \). Denote, \( P_1 = P_{p, \delta} \) and \( P_2 = P_{p, -\delta} \). The Kullback-Leibler divergence between the two satisfies for $\delta \le \frac{1}{2}$,
    \[
        \mathrm{KL}(P_1, P_2) = p \delta \log\frac{1 + \delta}{1 - \delta} \leq 4 p \delta^{2} .
    \]
    Using Pinsker's inequality and the additivity of the KL-divergence for product measures (see \citep{boucheron2013concentration}) we have
    \[
        \mathrm{TV}\left(P_1^{\otimes N}, P_{2}^{\otimes N}\right) \leq \sqrt{\frac{N}{2} \mathrm{KL}(P_1, P_2)}
        \leq \sqrt{2 N p \delta^{2}} = 1/2,
    \]
    where we choose \( \delta = 1 / \sqrt{8 N p} \) (the condition \( N \pmin > 1/8 \) ensures that \( \delta < 1/2\)). Below, we only consider the distributions \( P \in \{P_1, P_2\} \) instead of the whole class \( \mathcal{P}\left(\pmin, \sqrt{{5}/{8}}\right) \). Consider an empirical quantizer denoted as \( \hat{A}_{N} = \hat{A}_N(X_1, \dots, X_N)\) that takes only the values \( \{A_1, A_2\} \), where 
    \[
    A_1 = (-\tfrac{3}{4} p^{-1/2}, 0, \tfrac{1}{2} p^{-1/2}, p^{-1/2} ) \quad \text{and} \quad  A_2 = ( -p^{-1/2}, -\tfrac{1}{2} p^{-1/2}, 0, \tfrac{3}{4} p^{-1/2} ).
    \]
    Let \( \Omega_{1} \subset \R^{N} \) be the set such that \( \hat{A}_N = A_1 \) on it, and \( \hat{A}_N = A_2 \) outside of $\Omega_{1}$. Since \( \mathrm{TV}\left(P_1^{\otimes N}, P_2^{\otimes N}\right) \leq 1/2 \), we have
    \begin{align*}
        \max_{j = 1,2} P_j^{\otimes N}\left\{\hat{A}_N \neq A_j\right\} &= \max\left(1 - P_1^{\otimes N}(\Omega_1), P_2^{\otimes N}(\Omega_1)\right) 
        \\
        &\geq \max\left(1/2 - P_2^{\otimes N}(\Omega_1), P_2^{\otimes N}(\Omega_1)\right) \geq 1/4 .
    \end{align*}
    Notice that under \( \P_{j} \) the event \( \hat{A}_N \neq A_j \) corresponds to the distortion
    \[
        D(A_1, P_2) = D(A_2, P_1) = \frac{p (1 + \delta)}{2} \left(\frac{p^{-1/2}}{4}\right)^{2} = \frac{1 + \delta}{32},
    \]
    whereas the minimal distortion is 
    \[ 
        D(A_j, P_j) = \frac{p(1 - \delta)}{2} \left(\frac{p^{-1/2}}{4}\right)^{2} = \frac{1 - \delta}{32} \, .
    \]
    Since \( \delta = 1 / \sqrt{8 N p} \), \( p = 4 \pmin \), and \( \sigma^2 = 5/8\), with probability at least \(1/4\), we have 
    \[
        D(\hat{A}_N, P_j) - D(A_j, P_j) = \frac{\delta}{16} = \frac{1}{16 \sqrt{8 N p}} > \frac{\sigma^{2}}{60 \sqrt{N \pmin}} \, .
    \]

    It remains to show why only \( \hat{A}_N \in \{A_1, A_2\} \) matters. For an arbitrary \( \hat{A}_N \), the corresponding Voronoi cells could be one of the following:
    \begin{enumerate}
        \item \( \bigl\{\{ -p^{-1/2}, -\tfrac{1}{2}p^{-1/2}\}, \{0\}, \{\tfrac{1}{2}p^{-1/2}\}, \{p^{-1/2}\}\bigr\}, \)
        \item \( \bigl\{\{ -p^{-1/2} \}, \{-\tfrac{1}{2}p^{-1/2}, 0\}, \{\tfrac{1}{2}p^{-1/2}\}, \{p^{-1/2}\}\bigr\}, \)
        \item \( \bigl\{\{ -p^{-1/2} \}, \{-\tfrac{1}{2}p^{-1/2}\}, \{0, \tfrac{1}{2}p^{-1/2}\}, \{p^{-1/2}\}\bigr\}, \)
        \item \( \bigl\{\{ -p^{-1/2}\}, \{-\tfrac{1}{2}p^{-1/2}\}, \{0\}, \{\tfrac{1}{2}p^{-1/2}, p^{-1/2}\}\bigr\} \).
    \end{enumerate}
    Denote by \( \tilde{A}_{N} \) an empirical quantizer such that it equals to \(A_1\) in the cases~1. and~2., and equals to \(A_2\) in the cases~3. and~4. Let us show case by case, that the distortion of \( \tilde{A}_N \) is smaller under either measure.
    \begin{enumerate}
        \item This case is trivial: by the centroid condition under both measures the optimal center for the cluster \( \{ -p^{-1/2}, -\tfrac{1}{2}p^{-1/2} \} \) is \( -\tfrac{3}{4} p^{-1/2} \), which corresponds to \( A_1 \).
        \item It is easy to calculate that the minimal distortion of a cluster on two points \( a, b \) with probabilities \(q, r\), respectively, is \( (a - b)^{2} \tfrac{qr}{q + r} \). Therefore, using only the distortion on \( \{-\tfrac{1}{2}p^{-1/2}, 0\} \),
        \[
            D(\hat{A}_N, P_1) \geq \frac{1}{4p} \frac{p(1 - \delta) (1 - p)}{p(1 - \delta) + 4(1 - p)} > \frac{1 - \delta}{32} = D(A_1, P_1),
        \]
        where the second inequality follows from \( p \leq 2/5\). Using additionally \( \delta < 1\), we have
        as well
        \[
            D(\hat{A}_N, P_2) \geq \frac{1}{4p} \frac{p(1 + \delta) (1 - p)}{p(1 + \delta) + 4(1 - p)} >
            \frac{1 + \delta}{32} = D(A_1, P_2) \, .
        \]
    \end{enumerate}
    Due to the symmetry, case~3. is similar to case~2. and case~4. is similar to case~1. We conclude that we always have \( D(\hat{A}_N, P_j) \geq D(\tilde{A}_N, P_j) \) for both \(j = 1, 2 \).
\end{proof}

\section{Unknown Parameters of Distributions}
\label{sec:Generalcase}

In this section we show that a convergence rate similar to one in Theorem~\ref{thm:simpleupperm} and Theorem~\ref{thm:pmin} holds without any prior knowledge on $M$ or $\pmin$. Our motivation is that in practice we may not have any information about the underlying distribution $P$. We show that even in this case the sub-Gaussian excess distortion bounds are possible. However, as a result, our bounds become more sensitive to some specific properties of $P$. The following theorem is the main result of this section.

\begin{Theorem}
\label{thm:general_rate}
    Fix $\delta \in (0, 1)$.
    There is an estimator $\hat{A}_{\delta}$ depending on $\delta$ such that, with probability at least $1 - \delta$,
    \[
    D(\hat{A}_\delta) - D(A^*) \lesssim R\; \sqrt{\E \norm{X - \mu}^2} \left((\log N)^2 \sqrt{\frac{k}{N}} + \sqrt{\frac{\log\frac{1}{\delta}}{N}}\right),
    \]
    where $R$ is such that
    \begin{equation}\label{eq:R}
        \E \norm{X - \mu}^2 \Ind[\norm{X - \mu} > R] \le \frac{\Delta}{64},
    \end{equation}
    and
    \begin{equation}\label{def:Delta}
        \Delta = \inf_{A \in \A^{k-1}} D(A) - \inf_{A \in \A^k} D(A).
    \end{equation}
\end{Theorem}

\begin{Remark}
    Observe that both $R$ and $\Delta$ played an import role in the original proof of the strong consistency by \citeauthor{pollard1981strong}.
\end{Remark}

Let us first define our estimator. As before, in this section we use the notation~\eqref{def:Ak}.
\begin{framed}
    \textbf{The estimator of Theorem~\ref{thm:general_rate}.} We set
    \[
    \hat{A}_{\delta} = \argmin_{A \in \A^k} \MOM(l_{A}),
    \]
    with the number of blocks $\ell = 32 \left\lceil \log \tfrac{4}{\delta}\right\rceil + 1$.
\end{framed}
As before, our estimator $\hat{A}_{\delta}$ is an analog of an empirically optimal quantizer \eqref{empquantizer} with the only difference that instead of the sample mean we minimize the MOM criterion. 
Note that the estimator is translation invariant, so we can once again assume that \(\E X = 0\) without loss of generality.

\begin{proof}[Proof of Theorem~\ref{thm:general_rate}]
    The proof is based on the following simple observation: if for $A \in \A^k$ there is $a \in A$ such that \(\norm{a} \gg R\), then by considering \(A^{\prime} = A \setminus \{a\}\) we obtain that \(A' \in \A^{k-1}\) and \(D(A') - D(A) \ll \Delta\), thus $A$ cannot be a (nearly) optimal empirical solution.
    Namely, we are going to compare $\hat{A}_\delta$ with $\hat{A}_\delta \cap B_M$ for some $M \gtrsim R$ and show that with high probability either $\E l_{\hat{A}_\delta \cap B_M}$ is close to $\E l_{\hat{A}_\delta}$ (for small $N$) or \(\hat{A}_\delta \subset B_M\) (for large $N$), where $B_M$ is the ball of radius $M$ centred at the origin. 
    Moreover, \(\min_{a \in \hat{A}_\delta} \norm{a} \lesssim \sqrt{\E \norm{X}^2}\) with high probability,
    thus in both cases we can apply Lemma~\ref{lem:concentration_Mm} to obtain the convergence rate of the form
    \[
        D(\hat{A}_\delta) - D(A^*) 
        \lesssim M \sqrt{\E \norm{X}^2} \left((\log N)^2 \sqrt{\frac{k}{N}} + \sqrt{\frac{\log\frac{1}{\delta}}{N}}\right).
    \]
    First, according to the first part of Lemma~\ref{lem:optimizer_Mm}, with probability at least \(1 - e^{- \ell / 12} \ge 1 - \delta / 4\),
    \[
    \min_{a \in \hat{A}_\delta} \norm{a} \le m = 4 \sqrt{2 \E \norm{X}^2}.
    \]
    Let us define $M = m + 2 (R \lor m)$. Note that 
    \[
    \E \norm{X}^2 \le R^2 + \E \norm{X}^2 \Ind[\norm{X} > R] \le R^2 + \frac{\Delta}{64} \le R^2 + \frac{\E \norm{X}^2}{64},
    \]
    thus $R \ge 0.99 \sqrt{\E \norm{X}^2}$, which implies $M \simeq R$.

    Now fix $A \in \A^k$ such that \(\min_{a \in A} \norm{a} \le m\). Then by~\eqref{eq:norm_x} for any $a \in A$ (if it exists) such that \(\norm{a} > M\) and any $x \in V_a$, one has \(\norm{x} > \frac{M - m}{2} = R \lor m\), thus \(l_A \equiv l_{A \cap B_M}\) on $B_{R \lor m}$. Moreover, recall that for all $x \in E$,
    \[
        \min_{a \in A} \norm{a - x} \le m + \norm{x}, \quad
        \min_{a \in A \cap B_M} \norm{a - x} \le m + \norm{x},
    \]
    and thus
    \begin{align*}
        l_{A \cap B_M}(x) - l_A(x) 
        &= \min_{a \in A \cap B_M} \norm{a - x}^2 - \min_{a \in A} \norm{a - x}^2
        \\
        &\le (m + \norm{x})^2 \Ind[\norm{x} > R \lor m] 
        \le 4 \norm{x}^2 \Ind[\norm{x} > R].
    \end{align*}
    To simplify the notation by writing $\MOM(f(X))$ and $\QOM_{\alpha}(f(X))$ we mean $\MOM(f)$ and $\QOM_{\alpha}(f)$ respectively. Therefore, we have
    \begin{align*}
        \MOM(l_A) &\ge \MOM\left(l_{A \cap B_M}(X) - 4 \norm{X}^2 \Ind[\norm{X} > R]\right) \\
        &\ge \QOM_{1/4}(l_{A \cap B_M}) - 4 \QOM_{3/4}\left(\norm{X}^2 \Ind[\norm{X} > R]\right).
    \end{align*}
    The last term can be bounded by Lemma~\ref{lem:quant_ratio}: with probability at least \(1 - \delta / 4\),
    \[
        \QOM_{3/4} \left(\norm{X}^2 \Ind[\norm{X} > R]\right) 
        \le 8 \E \norm{X}^2 \Ind[\norm{X} > R] 
        \le \frac{\Delta}{8},
    \]
    thus
    \begin{equation}\label{eq:QOM_AM_bound}
        \QOM_{1/4}(l_{A \cap B_M}) \le \MOM(l_A) + \frac{\Delta}{2}.
    \end{equation}
    Further, we can assume without loss of generality\ that $A^*$ belongs to $\A^k_{M, m}$. Indeed, \(\min_{a \in A^*} \norm{a} \le m\) according to the proof of Theorem~\ref{thm:pmin}, and if \(A^* \not\subset B_M\), then \(|A^* \cap B_M| < k\) and hence
    \[
        \E l_{A^*} + \Delta \le \E l_{A^* \cap B_M} 
        \le \E l_{A^*} + 4 \E \norm{X}^2 \Ind[\norm{X} > R] 
        \le \E l_{A^*} + \frac{\Delta}{16},
    \]
    which is possible only if $\Delta = 0$. But in this case \(\norm{X} \le R \le M\) almost surely and \(|\supp(P)| \le k - 1\), thus one can choose \(A^* = \supp(P) \in \A^k_{M, m}\). 
    Lemma~\ref{lem:concentration_Mm} ensures that, with probability at least $1 - \delta / 2$, for any $\alpha \in \left\{\frac{1}{4}, \frac{1}{2}\right\}$ it holds that
    \begin{equation}\label{eq:l_concentration}
        \sup_{A \in \A_{M, m}^k} \abs{\E l_{A}(X) - \QOM_{\alpha}(l_{A})} 
        \le C R \sqrt{\E \norm{X}^2} \left((\log N)^2 \sqrt{\frac{k}{N}} + \sqrt{\frac{\log \frac{1}{\delta}}{N}}\right),
    \end{equation}
    where $C > 0$ is an absolute constant.
    Finally, we get the following lines of inequalities, which hold with probability at least $1 - \delta$,
    \begin{align*}
        \E l_{\hat{A}_\delta \cap B_M} 
        &\le \QOM_{1/4}(l_{\hat{A}_\delta \cap B_M}) + C R \sqrt{\E \norm{X}^2} \left((\log N)^2 \sqrt{\frac{k}{N}} + \sqrt{\frac{\log \frac{1}{\delta}}{N}}\right) \\
        &\le \MOM(l_{\hat{A}_\delta}) + \frac{\Delta}{2} + C R \sqrt{\E \norm{X}^2} \left((\log N)^2 \sqrt{\frac{k}{N}} + \sqrt{\frac{\log \frac{1}{\delta}}{N}}\right) \\
        &\le \MOM(l_{A^*}) + \frac{\Delta}{2} + C R \sqrt{\E \norm{X}^2} \left((\log N)^2 \sqrt{\frac{k}{N}} + \sqrt{\frac{\log \frac{1}{\delta}}{N}}\right) \\
        &\le \E l_{A^*} + \frac{\Delta}{2} + 2 C R \sqrt{\E \norm{X}^2} \left((\log N)^2 \sqrt{\frac{k}{N}} + \sqrt{\frac{\log \frac{1}{\delta}}{N}}\right).
    \end{align*}
    Now there are two possible cases.
    If 
    \[
    C R \sqrt{\E \norm{X}^2} \left((\log N)^2 \sqrt{\frac{k}{N}} + \sqrt{\frac{\log \frac{1}{\delta}}{N}}\right) \ge \frac{\Delta}{4},
    \]
    then 
    \[
    \E l_{\hat{A}_\delta} \le \E l_{\hat{A}_\delta \cap B_M} 
    \le \E l_{A^*} + 4 C R \sqrt{\E \norm{X}^2} \left((\log N)^2 \sqrt{\frac{k}{N}} + \sqrt{\frac{\log \frac{1}{\delta}}{N}}\right).
    \]
    Otherwise, we have
    \[
    C R \sqrt{\E \norm{X}^2} \left((\log N)^2 \sqrt{\frac{k}{N}} + \sqrt{\frac{\log \frac{1}{\delta}}{N}}\right) < \frac{\Delta}{4},
    \]
    then \(\hat{A}_\delta \subset B_M\): indeed, \(\E l_{\hat{A}_\delta \cap B_M} < \E l_{A^*} + \Delta\), now assume \(|\hat{A}_\delta \cap B_M| < k\), then
    \[
        \E l_{\hat{A}_\delta \cap B_M} 
        \ge \inf_{A \in \A^{k-1}} \E l_A 
        = \E l_{A^*} + \Delta,
    \]
    and we obtain a contradiction.
    Thus, \(\hat{A}_\delta \in \A^k_{M, m}\), and~\eqref{eq:l_concentration} again yields
    \begin{align*}
        D(\hat{A}_\delta) - D(A^*) & = \E l_{\hat{A}_\delta} - \E l_{A^*} \\ 
        & \le \E l_{\hat{A}_\delta} - \MOM(l_{\hat{A}_\delta}) - \E l_{A^*} + \MOM(l_{A^*}) \\ 
        & \lesssim R \sqrt{\E \norm{X}^2} \left((\log N)^2 \sqrt{\frac{k}{N}} + \sqrt{\frac{\log \frac{1}{\delta}}{N}}\right) \, .
    \end{align*}
\end{proof}

We conclude this section by comparing Theorem~\ref{thm:general_rate} to Theorem~2.2 presented in \citep*{biau2008performance}. The form of the latter result is somewhat similar to our excess distortion bound. However, the proof of Theorem~2.2 contains an inaccuracy which, to the best of our understanding, cannot be immediately fixed. The problem in the proof comes from the application of Corollary~2.1 in \citep{biau2008performance} which requires that the centres belong to the set $\A_{M}^k$ (which is called $\F_M^k$ there) and also that the observations $X_1, \ldots, X_N$ are in a bounded domain with probability one. The last fact does not hold for the unbounded distributions considered there (recall our Remark~\ref{importantremark}). Fortunately, with additional technical efforts and by replacing the empirically optimal quantizer with our MOM minimizer, we achieve the result even stronger in a manner than one claimed in Theorem~2.2 by \citeauthor{biau2008performance}

\section{Discussion}
\label{Discussions}

Finally, we discuss several previous results related to clustering for heavy-tailed distributions as well as directions for future work. 

The results of \cite*{brownlees2015empirical} are only presented for $k$-medians (where the distortion is defined as $D(A) = \E \min\limits_{a \in A} \norm{X - a}$ instead of $D(A) = \E \min\limits_{a \in A} \norm{X - a}^2$). However, we believe that their techniques, at least if applied straightforwardly, would require $\E \norm{X}^4 < \infty$. Similarly, \citeauthor{brecheteau2018robust} require $\E \norm{X}^2 < \infty $, but the targeted quantizer is different from ours. 
Our Theorem~\ref{thm:simpleupperm} only requires $\E \norm{X}^2 < \infty$ and is valid for any separable Hilbert space, whereas Theorem~11 in \citep{brownlees2015empirical} depends explicitly on the dimension and has a worse dependence on the $\log \frac{1}{\delta}$-term. 
The uniform bounds in \citep{telgarsky2013moment, bachem2017uniform} provide uniform convergence bounds under $\E \norm{X}^4 < \infty$ in $\R^d$ that cannot be immediately converted into the excess distortion bounds similar to ours. 
Since these uniform bounds are tuned to the analysis of empirically optimal quantizers, they obviously have a suboptimal dependence on the confidence parameter $\delta$.

A natural course of further research is to introduce some favorable assumptions on the distribution $P$ leading to the so-called \emph{fast rates} for the excess distortion. These are the excess distortion bounds scaling as $O\left(\frac{1}{N}\right)$ instead of $O\left(\frac{1}{\sqrt{N}}\right)$ which, of course, cannot be obtained for free \citep{antos2005improved}. By now, these assumptions and their analysis are well-understood in the bounded case (see \citep{levrard2015nonasymptotic} and references therein). 
Another interesting direction is to sharpen our bounds and make our robust algorithms more practical. As already mentioned, we believe that making some assumptions on $\pmin$ and thus restricting the sizes of clusters is somewhat more natural than assuming that $M$ is known in advance. Finally, it is natural to further extend our main results to the situation where the adversarial corruption of the observations is allowed. We refer to \citep{minsker2018uniform} for the related techniques.

\subsection*{Acknowledgements}
We would like to thank Olivier Bachem for stimulating discussions, G\'{a}bor Lugosi for a valuable feedback and Marco Cuturi and Nikita Puchkin for providing several important references. We are also thankful to the three anonymous referees for their useful comments and suggestions.

The work of Alexey Kroshnin was conducted within the framework of the HSE University Basic Research Program.
Results of Section 4 have been obtained under support of the RSF grant No. 19-71-30020.

\bibliography{mybib}

\section*{Appendix}

\subsection*{Proof of Lemma~\texorpdfstring{\ref{lem:momuniform}}{2.3}}
First, notice that \( \E f - \QOM_{\alpha}(f) = \QOM_{1 - \alpha}(\E f - f) \). Therefore, \( \sup_{f \in \F}  (\E f - \QOM_{\alpha}(f)) > x \) is equivalent to
\[
    \sup_{f \in \F} \frac{1}{\ell} \sum_{t = 1}^{\ell} \Ind[\E f - \tilde{f}_t > x] \geq \alpha,
\]
where \( \tilde{f}_t = \frac{\ell}{N} \sum_{i \in I_t} f(X_i) \). Using the idea of \cite{Mendelson:2015:LWC:2799630.2699439}, denote the function \( \phi(u) = (u - 1) \Ind[1 \le u \le 2] + \Ind[u \geq 2] \), so that \( \phi \) is $1$-Lipschitz, and \( \phi(u) \geq \Ind[u \geq 2] \). Then, the above event is included in the following event
\[
    \sup_{f \in \F} \frac{1}{\ell} \sum_{t = 1}^{\ell} \phi\left(\frac{2(\E f - \tilde{f}_t)}{x}\right) \geq \alpha \, .
\]
Next, we write the bounded difference inequality (see \citep{boucheron2013concentration}) since the summands in the above are independent and bounded by one. We have that, with probability at least \( 1 - e^{-2 \ell y^{2}}\),
\begin{align*}
    \sup_{f \in \F} \frac{1}{\ell} \sum_{t = 1}^{\ell} \phi\left(\frac{2(\E f - \tilde{f}_t)}{x}\right) \leq & \sup_{f \in \F} \frac{1}{\ell} \sum_{t = 1}^{\ell} \E\phi\left(\frac{2(\E f - \tilde{f}_t)}{x}\right) \\
    & + \E \sup_{f \in \F} \frac{1}{\ell} \sum_{t = 1}^{\ell} \left\{\phi\left(\frac{2(\E f - \tilde{f}_t)}{x}\right) - \E\phi\left(\frac{2(\E f - \tilde{f}_t)}{x}\right) \right\} \\
    & + y .
\end{align*}
For the first part, since \( \phi(u) \leq \Ind[u \geq 1] \) and using Chebyshev's inequality, we write
\begin{align*}
    \sup_{f \in \F} \frac{1}{\ell} \sum_{t = 1}^{\ell} \E \phi\left(\frac{2 (\E f - \tilde{f}_t)}{x}\right) 
    &\leq \sup_{f \in \F} \frac{1}{\ell} \sum_{t = 1}^{\ell} \P\left( \E f - \tilde{f}_t \geq x / 2\right)
    \\
    &\leq \sup_{f \in \F} \frac{\Var(\tilde{f}_t)}{(x/2)^{2}}
    = \sup_{f \in \F} \Var(f) \frac{4 \ell}{N x^2} \,.
\end{align*}
For the second part we use the symmetrization and contraction arguments of \cite{ledoux2013probability}, so that together
\begin{align*}
    &\E \sup_{f \in \F} \frac{1}{\ell} \sum_{t = 1}^{\ell} \left\{\phi\left(\frac{2(\E f - \tilde{f}_t)}{x}\right) - \E\phi\left(\frac{2(\E f - \tilde{f}_t)}{x}\right) \right\}
    \\
    &\quad\quad\leq
    2 \E \sup_{f \in \F} \frac{1}{\ell} \sum_{t = 1}^{\ell} \eps_{t} \phi\left(\frac{2(\E f - \tilde{f}_t)}{x}\right) 
    \\
    &\quad\quad\leq
    2 \E \sup_{f \in \F} \frac{1}{\ell} \sum_{t = 1}^{\ell} \eps_{t} \frac{2(\E f - \tilde{f}_t)}{x},
\end{align*}
where \( \eps_1, \dots, \eps_t \) are i.i.d.\ Rademacher signs.
Using the symmetrization argument again, we have
\[
    \E \sup_{f \in \F} \frac{1}{\ell} \sum_{t = 1}^{\ell} \eps_{t} \frac{2(\E f - \tilde{f}_t)}{x} \leq \frac{4}{x} \E \sup_{f \in \F} \frac{1}{N} \sum_{i = 1}^{N} \eps_{i} f(X_i) \, .
\]
Collecting the three terms together we have that, with probability at least \( 1 - e^{-2 \ell y^{2}}\),
\[
    \sup_{f \in \F} \frac{1}{\ell} \sum_{t = 1}^{\ell} \Ind\left[ \E f - \tilde{f}_t > x\right] 
    \leq y + \frac{8}{x} \E \sup_{f \in \F} \frac{1}{N} \sum_{i = 1}^{N} \eps_{i} f(X_i) + \frac{4}{x^2} \sup_{f \in \F} \Var(f) \frac{\ell}{N} .
\]
We need the right-hand side of the last display to be smaller than \( \alpha \). Let us take \( y = \alpha / 2 \) and
\[
    x = \frac{16}{\alpha} \E \sup_{f \in \F} \frac{1}{N} \sum_{i = 1}^{N} \eps_{i} f(X_i) + \sqrt{\frac{8}{\alpha} \sup_{f \in \F} \Var(f) \frac{\ell}{N}} .
\]
Then, with probability at least \( 1 - e^{-\alpha^{2} \ell / 2} \), it holds that
\[
    \sup_{f \in \F} \frac{1}{\ell} \sum_{t = 1}^{\ell} \Ind\left[ \E f - \tilde{f}_t > x\right] \leq \alpha,
\]
where we used that \(\frac{a}{a + b} + \frac{b^2}{(a + b)^2} \le 1\) for all $a, b > 0$.
To derive the other tail, we can use the symmetry \( \QOM_{\alpha}(f) = - \QOM_{1 - \alpha}(-f)\).
\qed

\subsection*{Proof of Lemma~\texorpdfstring{\ref{lem:quant_ratio}}{3.4}}
    By Markov's inequality and Chernoff's bound for the binomial distribution, we have
    \begin{align*}
        \P\left(\mathrm{Quant}_{1 - \alpha}(\xi_1, \dots, \xi_\ell) \ge 2 \E \xi / \alpha\right)
        &= \P\left(\sum_{i=1}^{\ell} \Ind[\xi_i \ge 2 \E \xi / \alpha] \ge \alpha \ell\right) 
        \\
        &\le \exp\left(- \frac{1}{3} \left(\alpha - \frac{\alpha}{2}\right) \ell\right)
        = \exp\left(- \frac{\alpha \ell}{6}\right).
    \end{align*}
\qed

\subsection*{Proof of Lemma~\texorpdfstring{\ref{lem:cov_number}}{3.6}}
    \textbf{Step 1.} 
    We start with $E = \R^d$.
    In order to prove the bound we observe that $L_{2}(P_N)$ distance between $l_A$, \(A = (a_1, \dots, a_k) \in \A^k_{M, m}\), and $l_B$, \(B = (b_1, \dots, b_k) \in \A^k_{M, m}\) (we can multiply some points if $|A| < k$ or $|B| < k$), is controlled by the maximum of the Euclidean distances between the corresponding vectors $a_j$ and $b_j$.
    Indeed, let $x \in V_{a_j} \cap V_{b_s}$, then the following assertions hold:
    \begin{align*}
        l_B(x) - l_A(x) &\le \norm{b_j}^2 - \norm{a_j}^2 - 2 \langle x, b_j - a_j \rangle \le 2 \left(\norm{a_j} + \norm{x}\right) \norm{a_j - b_j} + \norm{a_j - b_j}^2, \\
        l_A(x) - l_B(x) &\le \norm{a_s}^2 - \norm{b_s}^2 - 2 \langle x, a_s - b_s \rangle \le 2 \left(\norm{b_s} + \norm{x}\right) \norm{a_s - b_s} + \norm{a_s - b_s}^2.
    \end{align*}
    Therefore, 
    \[
    \abs{l_A(x) - l_B(x)} \lesssim \left(\norm{a_j} + \norm{b_s} + \norm{x}\right) \max_r \norm{a_r - b_r} + \max_r \norm{a_r - b_r}^2.
    \]
    On the other hand, 
    \[
    \norm{a_j - x} = \min_{a \in A} \norm{a - x} \le \norm{x} + \min_{a \in A} \norm{a} \le \norm{x} + m,
    \]
    as well as \(\norm{b_s - x} \le \norm{x} + m\), thus
    \[
    \abs{l_A(x) - l_B(x)} = \abs{\norm{a_j - x}^2 - \norm{b_s - x}^2} \le \left(\norm{x} + m\right)^2.
    \]
    Combining the above bounds and using the inequality \(u^2 \wedge v^2 \le u v\) for $u, v \ge 0$, we conclude that
    \begin{align*}
        \abs{l_A(x) - l_B(x)} 
        &\lesssim \left(\norm{a_j} + \norm{b_s} + \norm{x}\right) \max_r \norm{a_r - b_r} + (\norm{x} + m)^2 \land \max_r \norm{a_r - b_r}^2 \\
        &\lesssim \left(\norm{a_j} + \norm{b_s} + \norm{x}\right) \max_r \norm{a_r - b_r} + (\norm{x} + m) \max_r \norm{a_r - b_r} \\
        &\lesssim \left(\norm{a_j} + \norm{b_s} + \norm{x} + m\right) \max_r \norm{a_r - b_r}.
    \end{align*}
    Note that~\eqref{eq:A_radius} applied to the empirical measure $P_N$ ensures
    \begin{equation}\label{eq:sigma_N}
        \sum_{a \in A} \norm{a}^2 P_N(V_a) \le \sigma_N^2, \quad 
        \sum_{b \in B} \norm{b}^2 P_N(V_b) \le \sigma_N^2, \quad
        \text{where}\quad \sigma_N^2 = 2 m^2 + 8 P_N \norm{X}^2.
    \end{equation}
    Therefore, we have
    \begin{align*}
        &\norm{l_A - l_B}_{L_2(P_N)} 
        \\
        &\lesssim \left(\sqrt{\sum_{a \in A} \norm{a}^2 P_N(V_a)} + \sqrt{\sum_{b \in B} \norm{b}^2 P_N(V_b)} + \sqrt{\frac{1}{N} \sum_i \norm{X_i}^2} + m\right) \max_r \norm{a_r - b_r} \\
        &\le \left(2 \sigma_N + m + \sqrt{P_N \norm{X}^2}\right) \max_r \norm{a_r - b_r} \lesssim \sigma_N \max_r \norm{a_r - b_r}.
    \end{align*}
    Finally, we use that in $(\R^d)^k$ it holds that 
    \[
    \log \mathcal{N}_\infty\left((B_M)^k, t\right) \le k \log \mathcal{N}\left({B}_M, t\right) \lesssim k d \log \frac{M}{t}, 
    \]
    (see, e.g., \citep{Vershynin2016HDP}).
    
    \textbf{Step 2.} 
    Now we are ready to prove the bound in its full generality.
    The idea is to show that by the Johnson--Lindenstrauss lemma, there exists a low-dimensional subspace $L$ such that for a significant fraction of sets $A$ under consideration the functions $l_A$ remain almost the same if one projects both \(X_1, \dots, X_N\) and \(a_1, \dots, a_k \in A\) on $L$; then we apply the result of \textbf{Step~1}.
    
    First, note that \eqref{eq:var_Mm} implies (which holds for $P_N$ as well) 
    \[
    \norm{l_A - l_{\{0\}}}_{L_2(P_N)} \le 2 M \sqrt{6 P_N \norm{X}^2 + m^2} \le 2 M \sigma_N \quad\text{for all}\quad A \in \A^k_{M, m},
    \]
    where $\sigma_N$ comes from~\eqref{eq:sigma_N}, so it is enough to consider \(t \le 2 M \sigma_N\).
    
    We are going to apply the Johnson--Lindenstrauss lemma, and to do this, we first show that it is enough to consider quantizers $A$ from some finite-dimensional subspace of $E$, depending on the sample $X_1, \dots, X_N$.
    Let us fix an arbitrary vector \(u \notin \mathrm{Span}\left(\{X_1, \dots, X_N\}\right)\). It is easy to see that for any $a \in E$ there exists \(\tilde{a} \in S = \mathrm{Span}(\{u, X_1, \dots, X_N\})\) such that \(\norm{\tilde{a}} = \norm{a}\) and \(\langle \tilde{a}, X_i\rangle = \langle a, X_i\rangle\) for all $1 \le i \le N$. Therefore, without loss of generality, one can restrict $\A^k_{M, m}$ to the sets from the $(N+1)$-dimensional subspace $S$.
    Using the last observation, by the version of the Johnson--Lindenstrauss lemma for products (p.998 in \citep*{fefferman2016testing}) for any $0 < \eps \le 1/2$ and any fixed set $Q \subset S$ with \(|Q| \le N + k\), it holds that 
    \[
        \P_L\left(E^Q_L\right) 
        = \P_L\bigl(\abs{\langle R_L x, R_L y\rangle - \langle x, y\rangle} \le \eps \norm{x} \cdot \norm{y} ~~\text{for all}~~ x, y \in Q\bigr) 
        \ge \frac{1}{2}.
    \]
    Here $L$ is a random uniformly distributed $d$-dimensional subspace of $S$ with \(d = \left\lceil\frac{c \log(N + k)}{\eps^2}\right\rceil\) (we refer to \citep{johnson1984extensions} for a rigorous mathematical definition), and \(R_L = \sqrt{\frac{N + k}{d}} \Pi_L\), where $\Pi_L$ is the orthogonal projector on $L$.
    Let $\mathcal{P}(t) \subset \A^k_{M, m}$ be such that $\left\{l_A \;:\; A \in \mathcal{P}(t)\right\}$ is a $t$-packing set of $\F^k_{M, m}$ (i.e.,\ a maximal $t$-separated set), so that by the standard relation \(\mathcal{N}_2(\F^k_{M, m}, t, P_N) \le |\mathcal{P}(t)|\). 
    Now notice that for $A$ chosen uniformly at random from $\mathcal{P}(t)$ and random $L$ with joint probability at least $\frac{1}{2}$ the above condition holds for the set \(Q_A = \{X_1, \dots, X_N\} \cup A\):
    \[
    \P\left(E^{Q_A}_L\right) = \E_A \P_L\left(E^{Q_A}_L\right) \ge \frac{1}{2}.
    \]
    Note that we consider the union of the sample and a quantizer since we have to bound both the norms of the projections and the products of the form $\langle R_L X_i, R_L a\rangle$.
    On the other hand, by Fubini's theorem
    \[
    \P\left(E^{Q_A}_L\right) = \E_L \P_A\left(E^{Q_A}_L\right) \ge \frac{1}{2},
    \]
    therefore, there exists a subspace $L$ such that $\P_A\left(E^{Q_A}_L\right) \ge \frac{1}{2}$, that is, the event $E^{Q_A}_L$ holds for at least half of quantizers from $\mathcal{P}(t)$~--- let us denote this set by $\mathcal{P}_L(t)$.
    
    Consider an arbitrary quantizer $A \in \mathcal{P}_L(t)$. 
    In what follows we use for brevity the following simple notation: \(x' = R_L x\) for any $x \in E$ and, respectively, 
    \(A' = \left\{a' \;:\; a \in A\right\}\) and \(P_N' = \frac{1}{N} \sum_i \delta_{X_i'}\).
    Clearly, by the Johnson--Lindenstrauss lemma for any $x \in Q_A$,
    \[
        \frac{1}{2} \norm{x}^2 \le (1 - \eps) \norm{x}^2 
        \le \norm{x'}^2 = \langle R_L x, R_L x\rangle 
        \le (1 + \eps) \norm{x}^2 \le \frac{3}{2} \norm{x}^2,
    \]
    thus
    \[
        \max_{a' \in A'} \norm{a'}^2 \le \frac{3}{2} M^2, \quad 
        \min_{a' \in A'} \norm{a'}^2 \le \frac{3}{2} m^2, \quad
        P_N \norm{X'}^2 \le \frac{3}{2} P_N \norm{X}^2.
    \]
    In particular, this implies that 
    \[
    A' \in \A^k_{3M/2, 3m/2} \quad\text{and}\quad 
    \sum_{a' \in A'} \norm{a'}^2 P_N(V_{a'}) \le \frac{3}{2} \sigma_N^2,
    \]
    with $\sigma_N$ defined by~\eqref{eq:sigma_N}.
    Now for any \(X_i \in V_a(A)\), where $V_a(A)$ is the Voronoi cell from the partition induced by the set $A$, corresponding to the point $a$, we have
    \begin{align*}
        l_{A'}(X_i') \le \norm{a'}^2 - 2 \langle X_i', a'\rangle 
        &\le \norm{a}^2 - 2 \langle X_i, a\rangle + \eps \left(\norm{a}^2 + 2 \norm{X_i} \cdot \norm{a}\right)
        \\
        &\le l_A(X_i) + \eps M \left(\norm{a} + 2 \norm{X_i}\right),
    \end{align*}
    and in the same way we obtain that for \(X_i' \in V_{a'}(A')\),
    \begin{align*}
        l_A(X_i) &\le l_{A'}(X_i') + \eps \left(\norm{a}^2 + 2 \norm{X_i} \cdot \norm{a}\right)
        \\
        &\le l_{A'}(X_i') + \eps M \left(\norm{a} + 2 \norm{X_i}\right)
        \le l_{A'}(X_i') + \eps M \left(\sqrt{2} \norm{a'} + 2 \norm{X_i}\right).
    \end{align*}
    Therefore, recalling the definition of $\sigma_N$ \eqref{eq:sigma_N}, we have
    \[
    \norm{l_A(X) - l_{A'}(X')}_{L_2(P_N)} \le \eps M \left(2 \sigma_N + 2 \sqrt{P_N \norm{X}^2}\right) \le 3 \eps M \sigma_N.
    \]
    Setting \(\eps = \frac{t}{12 M \sigma_N} \le \frac{1}{6}\) we get \(\norm{l_A(X) - l_{A'}(X')}_{L_2(P_N)} \le \frac{t}{4}\), thus 
    \[
    \norm{l_{A'}(X) - l_{B'}(X)}_{L_2(P_N')} \ge \norm{l_A(X) - l_{B}(X)}_{L_2(P_N)} - \frac{t}{2} > \frac{t}{2} \quad\text{for any}\quad A \neq B \in \mathcal{P}_L(t). 
    \]
    This implies by the standard relation between the covering and packing numbers that \(|\mathcal{P}_L(t)| \le \mathcal{N}_2\left(\F', t/4, P_N'\right)\), where \(\F' = \left\{l_{A'} : A \in \A^k_{M, m}\right\}\).
    As was shown above \(\F' \subset \F^k_{3M/2, 3m/2}\), and since the corresponding quantizers belong to the $d$-dimensional subspace $L$, we have by \textbf{Step~1} that
    \[
        \log \mathcal{N}_{2}\left(\F', t/4, P'_N\right) 
        \lesssim k d \log \frac{M \sigma_N}{t}
        \lesssim \frac{k M^2 \sigma_N^2 \log(N + k)}{t^2} \log \frac{M \sigma_N}{t}.
    \]
    Combining our bounds, we conclude that
    \begin{align*}
        \log \mathcal{N}_2(\F^k_{M, m}, t, P_N) 
        &\le \log |\mathcal{P}(t)| 
        \le \log\left(2 |\mathcal{P}_L(t)|\right)
        \\
        &\lesssim \log \mathcal{N}_2\left(\F', t/4, P_N'\right) 
        \lesssim \frac{k M^2 \sigma_N^2 \log(N + k)}{t^2} \log \frac{M \sigma_N}{t}.
    \end{align*}
    To obtain the claimed bound, it remains to notice that for \(N < k\), it is enough to consider the class $\A^N_{M, m}$ instead of $\A^k_{M, m}$, thus we can always assume \(k \le N\). Hence,
    \[
        \log \mathcal{N}_2(\F^k_{M, m}, t, P_N) 
        \lesssim \frac{k M^2 \sigma_N^2 \log(2 N)}{t^2} \log \frac{M \sigma_N}{t}.
    \]
    The claim follows.
\qed
\end{document}